\documentclass[12pt]{amsart}

\usepackage{graphicx}
\usepackage{color}
\usepackage{amssymb}
\usepackage{enumitem}
\usepackage{hyperref}
\usepackage{comment}

\newcommand{\R}{\mathbb{R}}
\newcommand{\C}{\mathbb{C}}

\newcommand{\vectornorm}[1]{\left\|#1\right\|}

\newcommand{\acs}[1]{#1}
\newcommand{\J}{\acs{J}}
\newcommand{\restrict[3]}{\left. {#2}\right|_{#3}}
\newcommand{\VN}[2]{\vectornorm{#1}^{#2}}
\newcommand{\VNC}[3]{\vectornorm{#1}^{#2}_{#3}}

\newcommand{\Dl}[1]{D_{#1}^Y}
\newcommand{\Dr}[1]{D_{#1}^{\nu_Y}}
\DeclareMathOperator{\id}{Id}
\DeclareMathOperator{\hess}{Hess}

\setlist[enumerate,1]{label={(\alph*)}}
\setlist[enumerate,2]{label={(\roman*)}}

\newtheorem{tm}{Theorem}[section]

\newtheorem{lm}[tm]{Lemma}
\newtheorem{cy}[tm]{Corollary}
\newtheorem{cm}[tm]{Claim}

\theoremstyle{definition}
\newtheorem{df}[tm]{Definition}

\makeatletter
\def\@opargbegintheorem#1#2#3{\trivlist
       \item[\hskip \labelsep{\bfseries #1\ #2}]{\bfseries(#3)\ }\itshape}
\makeatother

\theoremstyle{remark}
\newtheorem{rem}[tm]{Remark}

\title{J-holomorphic curves with boundary in bounded geometry}
\author{Yoel Groman and Jake P. Solomon}
\date{Oct. 2014}
\begin{document}
\begin{abstract}
The fundamental properties of $J$-holomorphic curves depend on two inequalities: The gradient inequality gives a pointwise bound on the differential of a $J$-holomorphic map in terms of its energy. The cylinder inequality stipulates and quantifies the exponential decay of energy along cylinders of small total energy. We show these inequalities hold uniformly if the geometry of the target symplectic manifold and Lagrangian boundary condition is appropriately bounded.
\end{abstract}
\maketitle
\tableofcontents
\section{Introduction}
Let $(M,\omega,L,J)$ be a symplectic manifold with Lagrangian submanifold $L$ and almost complex structure $J$ which is tamed by $\omega$. Denote by $g_J$ the symmetrization of the positive definite form $\omega(\cdot,J\cdot)$. Consider first the case where $M$ and $L$ are compact. It is shown in \cite[Ch. 4]{MS2} that there are constants $c_1,\delta_1,$ such that the following gradient inequality holds. Let $\mathbb{H} \subset \C$ denote the upper half-plane, and for $r>0$ define $U_r:=B_{r}\cap\mathbb{H}.$ Let
\[
u:(U_{2r}, U_{2r}\cap\R)\to (M,L)
\]
be $J$-holomorphic. Write $E(u;U_{r}):=\frac12\int_{U_r}|du|^2$. Then
\begin{equation}\label{IntGradIneq}
E(u;U_{2r})<\delta_1\qquad\Rightarrow\qquad\sup_{U_r}|du|^2\leq\frac{c_1}{r^2}E(u;U_{2r}).
\end{equation}

A further basic estimate shown in \cite[Ch. 4]{MS2} is the cylinder inequality. It states that there are constants $c_2, c_3$, and $\delta_2$ such that the following holds. Denote by $I_R$ the cylinder $[-R,R]\times S^1$. For any $J$-holomorphic map $u:I_R \to M$ and for any $T\in[c_2,R]$, we have
\begin{equation}\label{IntCylIneq}
E(u;I_R)<\delta_2\qquad\Rightarrow\qquad E(u;I_{R-T})\leq e^{-c_3T}E(u;I_R).
\end{equation}
There are variants for the case of strips and cylinders with Lagrangian boundary conditions. See Theorem~\ref{tmCylEq} below.

The aim of the present paper is to establish sufficient conditions for the above inequalities to hold when $M,L,$ are not necessarily compact.

Denote by $R$ the curvature of $M,$ by $B$ the second fundamental form of $L$ and by $i$ the radius of injectivity of $M$, all with respect to the metric $g_J$. For a tensor $A$ on $M$ or $L$ we denote by $\vectornorm{A}_m$ the $C^m$ norm of $A$ with respect to $g_J$. For any Riemannian manifold $X$ with submanifold $Y$ and $\epsilon>0$, we say that $Y$ is $\epsilon$-Lipschitz if
\[
\frac{d_X(x,y)}{\min\{1,d_Y(x,y)\}} \geq\epsilon\qquad \forall x\neq y\in Y.
\]
We say that $Y$ is Lipschitz if there is an $\epsilon$ such that $Y$ is $\epsilon$-Lipschitz.
\begin{tm}\label{MainTheorem}
Suppose $M$ and $L$ are complete with respect to $g_J$, $J|_L$ is compatible with $\omega$, $L$ is Lipschitz, and
\[
\max\left\{\vectornorm{R}_2,\vectornorm{J}_3,\vectornorm{B}_2,\frac1{i}\right\}<\infty.
\]
Then there are constants $c_1,~c_2,~c_3,~\delta_1$ and $\delta_2,$ such that inequalities~\eqref{IntGradIneq} and \eqref{IntCylIneq} hold.
\end{tm}

In \cite{ThTh} and \cite{GS12} we study quantitative aspects of the geometry of $J$-holomorphic curves in $M$ with boundary on $L$, based on these estimates. One application of such quantitative results is for proving Gromov compactness in settings where $M$ and $L$ are not compact.

The difficult part of the proof is the gradient inequality~\eqref{IntGradIneq}. In the compact case this is proven in \cite{MS2} based on a reflection construction by Frauenfelder \cite{Frf}. In this construction the metric $g_J$ is altered in a neighborhood of $L$ in such a way that $L$ becomes totally geodesic while the new metric remains Hermitian with respect to $J$. A large part of this paper is devoted to the proof of Theorem~\ref{tmLagTot}, which states that given the bounds of Theorem~\ref{MainTheorem}, the reflection construction can be done while preserving the boundedness of curvature and the first two derivatives of $J$. See Section~\ref{PrftmLagTot}.

The assumption that $J|_L$ is compatible with $\omega$, not necessary in the compact case, is important for us because it implies that $JTL$ is orthogonal to $TL$. Compatibility can likely be replaced by quantitative assumptions on $J|_L$ such as a bound from below on the angle between $TL$ and $JTL$. However, compatibility along $L$ does not seem to be a serious restriction. For example, tame almost complex structures are used to prove that Gromov-Witten type invariants are invariant under deformations of the symplectic form $\omega.$ But deformations of the symplectic form along $L$ are trivial by the Weinstein neighborhood theorem.

Basic estimates on the Riemannian geometry of submanifolds are proved in the non-compact case. Theorem~\ref{tmTubWidth} gives a lower bound on the distance to the cut locus of a submanifold. Lemma~\ref{lmInjRadEst} gives a lower bound on the injectivity radius of a submanifold. In both cases, these results were previously known in the compact case~\cite{Singh,Corlette90}.

In the text we use the notion of a thick-thin measure on a Riemann surface $\Sigma$ with boundary. This notion is a more intrinsic formulation of the inequalities~\eqref{IntGradIneq} and \eqref{IntCylIneq} satisfied by the energy measure induced on $\Sigma$ by a $J$-holomorphic map $u:(\Sigma,\partial\Sigma)\to (M,L)$. See Section~\ref{sec:thth} for details. Our motivation in introducing this formulation is twofold. First, it is sometimes useful to apply conformal changes to the metric on the domain, which calls for a conformally invariant formulation of the gradient inequality. Second, we can slightly weaken the Lipschitz condition on $L$ in Theorem~\ref{MainTheorem} and replace it with a condition on $\Sigma$. For this we need a formulation that refers to the domain as a whole. The full statement of our results is given in Theorem~\ref{EnThTh}.

A natural question that is not discussed here concerns the topology of the space of $\omega$ tame almost complex structures that satisfy the hypothesis of Theorem~\ref{MainTheorem}. In particular, what conditions on two such almost complex structures guarantee they belong to the same connected component? We leave this for future research.

\subsection{Acknowledgements}
The authors were partially supported by ERC Starting Grant 337560 and ISF Grant 1747/13. Y.G. is grateful to the Azrieli foundation for the award of an Azrieli fellowship. The authors would like to thank the referee for many helpful comments and suggestions.

\section{Thick-thin measure}\label{sec:thth}
\subsection{Preliminaries on conformal geometry}\label{SecConfGeom}
\begin{df}\label{dfSubCyinlder}
Let $(I,j)$ be a compact doubly connected surface with complex structure $j$. The \textbf{modulus} of $(I,j)$, denoted by $Mod(I,j)$ or $Mod(I)$ when the complex structure is clear from the context, is the unique real number $r>0$ such that $(I,j)$ is conformally equivalent to $[0,r]\times S^1$. Here $S^1$ is taken to be a circle of length $2\pi$. Equip $I=[a,b]\times S^1$ with the product orientation and coordinates $\rho,\theta,$ for the factors $[a,b], S^1,$ respectively. Define a metric $h$ on $I$ by
\begin{equation}\label{eq:axsy}
h =d\rho^2+h_{\theta}(\rho)^2d\theta^2,
\end{equation}
and let $j$ be the induced complex structure. Then
\begin{equation}\label{EqmodFrom}
Mod(I,j)=\int_a^b\frac1{h_{\theta}(\rho)}d\rho.
\end{equation}

Let $I$ be a doubly connected compact surface, and let $L:=Mod(I)$. Then there is a holomorphic map $f:[0,L]\times S^1\rightarrow I$ unique up to a rotation and a holomorphic reflection. Fix one such $f$. For real numbers $a\leq b\in [0,L]$, we write
\[
S(a,b;I):=f([a,b]\times S^1)\subset I.
\]
For $a,b \in [0,L]$ with $a \leq L-b$, we write
\[
C(a,b;I):=S(a,L-b;I).
\]
Note that composing $f$ with the holomorphic reflection of $[0,L]\times S^1$ replaces $S(a,b)$ with $S(L-b,L-a)$. The expression $C(a,a)$, however, is independent of the choice of $f$.
\end{df}

\begin{df}\label{DfCofRad}
Let $U$ be a Riemann surface biholomorphic to the unit disk $D_1$. Let $h$ be a conformal metric on $U$ and let $z\in U$. Then there is a biholomorphism $\phi:U\rightarrow D_1$ with $\phi(z)=0,$ unique up to rotation. The \textbf{conformal radius of $U$ viewed from $z$} is defined to be
\[
r_{conf}(U,z;h):=1/\|d\phi(z)\|_h.
\]

\end{df}

\begin{df}\label{dfCplxDbl}
For any Riemann surface $\Sigma=(\Sigma,j)$, write $\overline{\Sigma}:=(\Sigma,-j)$. The \textbf{complex double} of $\Sigma$ is the Riemann surface
\[
 \Sigma_{\C} :=\Sigma\cup\overline{\Sigma},
\]
where the surfaces are glued together along the boundary by the identity map. The complex structure on $\Sigma_{\C}$ is the unique one which coincides with $j$ and with $-j$ when restricted suitably. $\Sigma_{\C}$ is endowed with a natural anti-holomorphic involution, and for any $z\in\Sigma_{\C}$ we denote by $\overline{z}$ the image of $z$ under this involution.
\end{df}

\begin{df}
Let $\Sigma$ be a connected Riemann surface. A subset $S\subset\Sigma_{\C}$ is said to be \textbf{clean} if either $S=\overline{S}$ or $S\cap \overline{S}=\emptyset$.
\end{df}
\subsection{Thick-thin measure}
For the rest of the discussion, fix constants $c_1,c_2,c_3,\delta_1,\delta_2>0$ such that $c_3\leq 1$ and that $\delta_2<\frac12\delta_1$. For a Riemann surface with metric $h,$ denote by $\nu_h$ the volume form of $h.$ For $\mu$ an absolutely continuous measure on $\Sigma,$ denote by $\frac{d\mu(z)}{d\nu_{h}}$ the Radon-Nikodym derivative with respect to $\nu_h$.

\begin{df}\label{dfThTh}
Let $(\Sigma,j)$ be a Riemann surface, possibly bordered. Let $\mu$ be a finite measure on $\Sigma$ and extend $\mu$ to a measure on $\Sigma_{\C}$ by reflection. That is,
\[
\mu(U):=\mu(\overline{U}),
\]
for $U\subset\overline{\Sigma}$  a measurable set. Suppose further that $\mu$ is absolutely continuous and has a continuous density $\frac{d\mu}{d\nu_h}$,  where $h$ is any Riemannian metric on $\Sigma_{\C}$.

The measure $\mu$ is said to satisfy the \textbf{gradient inequality} if the following holds. Let $U\subset\Sigma_{\C}$ be biholomorphic to the unit disk such that $U \cap \partial\Sigma$ is connected, and let $z\in U$. Then for any conformal metric $h$ on $(\Sigma_{\C},j)$,
\begin{align}\label{GradEq}
 \mu(U)<\delta_1\quad \Rightarrow \quad &\frac{d\mu}{d\nu_h}(z)\leq c_1\frac{\mu(U)}{r_{conf}^2},
\end{align}
where $r_{conf}=r_{conf}(U,z;h)$.

The measure $\mu$ is said to satisfy the \textbf{cylinder inequality} if the following holds. Let $I\subset\Sigma_{\C}$ be clean and doubly connected such that $Mod(I)> 2c_2$. Then for all $t\in \left(c_2,\frac1{2}Mod(I)\right)$ we have,
\begin{gather}
\notag \mu(I)<\delta_2 \qquad \Rightarrow \qquad\mu(C(t,t;I))\leq e^{-c_3t}\mu(I)\notag.
\end{gather}

The measure $\mu$  will be called \textbf{thick-thin} if it satisfies the gradient and cylinder inequalities.
\end{df}

\begin{df}\label{UniDfThTh}
A family of measured Riemann surfaces which are thick-thin with respect to given constants $c_i$, $\delta_i,$ will be referred to as a uniformly thick-thin family.
\end{df}

\subsection{Conventions}
For the rest of the paper, we fix a smooth symplectic manifold $(M,\omega)$ with Lagrangian submanifold $L$ and $\omega$-tame almost complex structure $J$. We assume further that $J|_L$ is compatible with $\omega$. Let $g$ be a Hermitian metric on $M$. Let $\Sigma$ be a compact Riemann surface, let $u:\Sigma\rightarrow M$ be a $\J$-holomorphic curve and let $h$ be a conformal metric on $\Sigma$. Let $z\in \Sigma.$ We denote by $\vectornorm{du(z)}_{g,h}$ the norm of $du(z)$ with respect to the metrics $g$ and $h$. The expression $\vectornorm{du}_{g,h}^2\nu_h$ is independent of the metric $h$. However, it does depend on $g$. Define the energy measure $\mu_{u,g}$ of $u$ with respect to $g$ by
\[
\mu_{u,g}(U):=\frac{1}{2}\int_{U}\vectornorm{du}_{g,h}^2\nu_h.
\]
When $g=g_J$ we omit $g$ from the subscript. We have \cite{MS2} $\vectornorm{du}_{h}^2\nu_h=u^*\omega$. So,
\begin{equation}\label{EqEnergyId}
\mu_u(U)=\int_{U}u^*\omega.
\end{equation}

Given any metric $g$ on a Riemannian manifold $X$, denote by $R^g$, $Sec_g$, $\nabla^g$ and $\exp^g$ respectively,  the curvature tensor of $g$, the sectional curvature of $g$, the Levi-Civita connection of $g$ and  the exponential map with respect to $g$. In this section, when the superscript is omitted and $X=M$ we refer to the metric $g=g_J$. For a tensor $H$ of type $(r,s)$ and $p \in M,$ we denote by $\VN{H_p}{g}$ the norm of $H_p$ with respect to the metric induced by $g$ on the tensor bundle of type $(r,s).$ We write
\[
\VN{H}{g}:=\sup_{p\in M}\VN{H_p}{g}.
\]
Note that for any $r$ vectors $v_1,...,v_r\in T_pM$ we have
\[
\vectornorm{H_p(v_1,...,v_r)}\leq \vectornorm{v_1}...\vectornorm{v_r}\vectornorm{H_p}.
\]
For $j\in\mathbb{N}$ we write
\[
\VNC{H}{g}{j} :=\sum_{i=0}^j\VN{\nabla^{g(j)}H}{g},
\]
where $\nabla^{g(0)}H:= H$ and $\nabla^{g(j+1)}H:=\nabla^g\nabla^{g(j)}H$.

We write $d(\cdot,\cdot;X,g)$, $\ell(\cdot;X,g)$, $Area(\cdot;X,g)$ and $InjRad(X,g)$ to denote distance, length, area, and radius of injectivity with respect to the Riemannian metric $g$ on $X$. We shall omit $X$ from the notation when $X$ is clear from the context. For $X=M$ we shall omit $g$ from the notation when $g=g_J$. When $X=L$ we shall omit it if $g$ is the induced metric $g_J|_L$.

Denote by $\pi :\nu_L \to L$ the normal bundle with respect to $g_J$. Denote by $O$ the zero section of $\nu_L$, and denote by
\[
B : TL \otimes TL \to \nu_L
\]
the second fundamental form of $L$ with respect to $g_J$. The expression $\vectornorm{B}_j$ denotes the $C^j$ norm with respect to the induced metric and connection on $T^*L\otimes T^*L\otimes TM\big|_L.$

Henceforth we shall always assume that $g_J$ and the induced metric on $L$ are complete.

\subsection{The inequalities for the energy distribution}
\begin{df}\label{BoundCond}
Let $S$ be a family of compact Riemann surfaces, possibly with boundary. We say that the data of $S$ together with $(M,\omega,L,J)$ comprise a \textbf{bounded setting} if one of the following holds.
\begin{enumerate}
\item\label{BoundCond1} $M$ and $L$ are compact.
\item \label{BoundCond2}$L=\emptyset$ and
\[
\max\left\{\vectornorm{R},\vectornorm{J}_2,\frac1{InjRad(M;g_J)}\right\}<\infty.
\]
\item\label{BoundCond3}
$L$ is Lipschitz and
\[
\max\left\{\vectornorm{R}_2,\vectornorm{J}_3,\vectornorm{B}_2,\frac1{InjRad(M;g_J)}\right\}<\infty.
\]
\item\label{BoundCond4}
Each connected component $L'$ of $L$ is Lipschitz and
\[
\max\left\{\vectornorm{R}_2,\vectornorm{J}_3,\vectornorm{B}_2,\frac1{InjRad(M;g_J)}\right\}<\infty.
\]
Furthermore, for each $\Sigma \in S$, there is a conformal metric $h$ of constant curvature $0,\pm 1,$ and of unit area in case of zero curvature, such that $\partial\Sigma$ is totally geodesic and $\epsilon$-Lipschitz.

\end{enumerate}
\end{df}

Let $\mathcal{F}$ be a family of $J$-holomorphic curves in $M$ with boundary in $L$. We associate with $\mathcal{F}$ the family $S_{\mathcal{F}}$ of domains of elements of $\mathcal{F}$ and the family $\tilde{\mathcal{F}}$ of measured Riemann surfaces
\[
\tilde{\mathcal{F}}:=\{(\Sigma,\mu_u)|[u:(\Sigma,\partial\Sigma)\to (M,L)]\in{\mathcal{F}}\}.
\]
\begin{tm}\label{EnThTh}
If $S_{\mathcal{F}}$ together with $(M,\omega,L,J)$ comprise a bounded setting, then $\tilde{\mathcal{F}}$ is uniformly thick-thin.
\end{tm}

The proof of Theorem~\ref{EnThTh}, based on results formulated in the next several pages, is given at the end of this section. For the proof of the gradient inequality in the noncompact bordered setting we need the following theorem whose proof is postponed to Section \ref{PrftmLagTot}. It is in the proof of this theorem that the bounds on the derivatives of the curvature and second fundamental form are required.
\begin{tm}\label{tmLagTot}
Suppose $(M,\omega,L,J)$ satisfy the bounds appearing in part~\ref{BoundCond3} of Definition~\ref{BoundCond}. Then there is a Hermitian metric $h$ on $M$ which satisfies the following conditions:
\begin{enumerate}
\item $h$ is norm equivalent to $g_J$.
\item $\max\left\{\vectornorm{R^h}^h, \vectornorm{J}^h_2\right\} <\infty.$
\item $L$ is totally geodesic with respect to $h$.
\item $\J TL=TL^{\perp}$.
\end{enumerate}
\end{tm}

The cylinder inequality relies on the isoperimetric inequality formulated and proven in the compact case in \cite[Remark~4.4.2]{MS2}. We recall the formulation. Let $\gamma:S^1\rightarrow M$ be a smooth loop satisfying
\[
\ell(\gamma)<InjRad(M).
\]
Let $B\subset\C$ be the unit disk. We have that $\gamma$ is contained in a geodesic ball $B'\subset M$ with radius $\frac12InjRad(M)$. Therefore $\gamma$ may be extended to a map $u_{\gamma}:B\rightarrow B'\subset M$ satisfying
\[
u_{\gamma}(e^{i\theta})=\gamma(\theta)
\]
for all $\theta\in S^1=\mathbb{R}/2\pi\mathbb{Z}$. We use this to define the symplectic action of $\gamma$ as
\begin{equation}
a(\gamma):=-\int_Bu^*_{\gamma}\omega.
\end{equation}
Note that this definition is independent of the choice of the extension so long as it is contained in any geodesic ball of radius $\frac12InjRad(M)$.

Analogously, let $\gamma:[0,\pi]\rightarrow M$ with $\gamma(\{0,\pi\})\subset L$. Suppose $L$ is $\epsilon$-Lipschitz, and let
\[
\delta= \epsilon\min\left\{1,\frac12 InjRad(M), \frac1{2}InjRad(L)\right\}.
\]
Suppose $\ell(\gamma)<\delta$. Then there is a path $\alpha:[0,1]\rightarrow L$ with
\begin{equation}\label{eqininj}
\ell(\alpha)<\min\left\{\frac12 InjRad(M), \frac1{2}InjRad(L)\right\}
\end{equation}
and $\alpha(x)=\gamma(x)$ for $x\in\{0,\pi\}$. Indeed, the estimate on $\ell(\gamma)$ and the $\epsilon$-Lipschitz condition imply that $\gamma(0)$ and $\gamma(\pi)$ are on the same connected component of $L$ and if $\alpha$ is taken to be a minimizing geodesic, estimate \eqref{eqininj} holds. Let $\tilde{\gamma}$ be the loop obtained from the concatenation of $\gamma$ and $\alpha$. We define the action of $\gamma$ as
\begin{equation}
a(\gamma):=a(\tilde{\gamma}).
\end{equation}
It follows from estimate \eqref{eqininj} that this definition is independent of the choice of $\alpha$.
\begin{lm}
There is a constant $c=c(\vectornorm{R})$ such that
\begin{equation}\label{IsopIneq}
\ell(\gamma)<\delta\qquad\Rightarrow \qquad |a(\gamma)|\leq c\ell(\gamma)^2.
\end{equation}
\end{lm}
\begin{proof}
In the compact setting, this is proven in \cite[Remark 4.4.2]{MS2}. The proof extends to the noncompact case if the curvature is bounded. A similar claim is shown in the proof of Proposition 4.7.2 and the comments after Definition 4.1.1 in \cite{Si}.
\end{proof}
Let $a>\pi$, and let $I$ be one of the following domains:
	\begin{enumerate}
		\item \label{RefCylCaseA}$(-a,a)\times S^1,$
		\item \label{RefCylCaseB}$(-a,a)\times [0,\pi],$
		\item \label{RefCylCaseC}$[0,a)\times S^1.$
	\end{enumerate}
	Let $u:(I,\partial I)\rightarrow (M,L)$ be $\J$-holomorphic. In cases~\ref{RefCylCaseB} and~\ref{RefCylCaseC}, extend the measure $\mu_u$ to $I_{\C}=(-a,a)\times S^1$ by reflection. Write $E:=\mu_u(I_{\C}).$
\begin{tm}\label{tmCylEq}Compare \cite[Lemma~4.7.3]{MS2}. Suppose $(I_\C,\mu_u)$ satisfies the gradient inequality, $InjRad(M) > 0$, $\|R\| < \infty,$ $L$ is $\epsilon$-Lipschitz and $InjRad(L)>0.$
There are constants $c_3$ and $\delta_2>0$, depending on the constants $c,c_1,\delta$, and $\delta_1$ of the isoperimetric and gradient inequalities respectively, such that for all $t \geq 2\pi,$ we have
\begin{align}\label{CylEstPr}
&E<\delta_2\qquad\Rightarrow\qquad \mu_u(C(t,t;I_{\mathbb{C}}))\leq e^{-c_3t}E.
\end{align}
\end{tm}
\begin{proof}
Take
\[
\delta_2:=\min\left\{\delta_1,\frac{\delta^2}{16c_1}\right\}.
\]
Suppose $E < \delta_2$. In cases~\ref{RefCylCaseA} and \ref{RefCylCaseC}, let $X= S^1$. In case~\ref{RefCylCaseB}, let $X=[0,\pi]$. For any $t\in[-a,a]$ write
\[
\gamma_t:=u|_{\{t\}\times X}.
\]
Let $z=(t,s)\in C(\pi,\pi;I_{\mathbb{C}})$. Then by the gradient inequality with $h$ the standard flat cylindrical metric on $I_{\mathbb{C}}$ of circumference $2\pi$, we have
\[
\vectornorm{du(z)}^2<\frac{4c_1\delta_2}{\pi^2}.
\]
Therefore,
\begin{align}\label{LengthEst}
\ell(\gamma_t)<\delta.
\end{align}
We treat first the cases~\ref{RefCylCaseA} and~\ref{RefCylCaseB}. By definition of $\delta$ capping off $u(C(t,t;I))$ with discs contained in geodesic balls produces a contractible sphere in case~\ref{RefCylCaseA} and a contractible disc relative to $L$ in case~\ref{RefCylCaseB}. Write
 \begin{align}
\notag \varepsilon (t):=\mu_u(C(t,t;I_{\mathbb{C}})).
\end{align}
By the energy identity, \eqref{EqEnergyId}, and the definition of the symplectic action, we have the equation
\begin{align}
\varepsilon (t) + a(\gamma_{a-t}) - a(\gamma_{-a+t})=0.
\end{align}
Using the isoperimetric inequality, we obtain
\begin{align}\label{eqEps}
\varepsilon(t) & =a(\gamma_{-a+t})-a(\gamma_{a-t})\\&\leq c(l(\gamma_{a-t})^2+l(\gamma_{-a+t})^2)
\notag\\&\leq 2\pi c \int_0^{2\pi}|\partial_su(a-t, s)|^2ds + 2\pi c \int_0^{2\pi}|\partial_su(-a+t,s)|^2ds\notag\\&=-2\pi c \dot{\varepsilon}(t)\notag.
\end{align}
Integrating this differential inequality from $\pi$ to $t$ we get the estimate
\[
\varepsilon(t)\leq e^{\frac{\pi-t}{2\pi c}}\varepsilon(\pi)\leq e^{\frac{\pi-t}{2\pi c}} E.
\]
This gives equation~\eqref{CylEstPr} with $c_3=\frac1{4\pi c}$ for cases~\ref{RefCylCaseA} and~\ref{RefCylCaseB}.

In case~\ref{RefCylCaseC}, note that $a(\gamma_0)=0$ since $L$ is Lagrangian and $\gamma_0$ is contractible in $L$ by estimate \eqref{LengthEst}. Define
\begin{align}
\notag \varepsilon(t):=\mu_u(S^1\times[0,a-t]), && 0<t< \pi.
\end{align}
Then $\varepsilon(t)=\mu_u(C(t,t;I_{\mathbb{C}}))/2$. Applying the same derivation as ~\eqref{eqEps} but dropping the second term in each line, we obtain estimate~\eqref{CylEstPr}.
\end{proof}

\begin{lm}\label{lmInjRadEst}
Suppose $|Sec|<C,$ $\vectornorm{B}<H$ and
\[
InjRad(M)>i_0
\]
for some positive constants $C,H,$ and $i_0$. Then
\[
InjRad(L)>0.
\]
\end{lm}
\begin{rem}\label{rmCorlette}
The same quantitative dependence of  $InjRad(L)$ on the geometry of $(M,L)$ in the compact case appears in the literature \cite{Corlette90}. However, the proof relies on compactness in an essential way. Namely, in the compact case, we have
\[
InjRad(L)\geq\min\left\{\frac{\pi}{\sqrt{K}},\frac12\ell(\gamma_{min})\right\},
\]
where $K$ bounds the sectional curvature of $L$ and $\gamma_{min}$ is the smallest closed geodesic in $L$.  A lower bound on $\ell(\gamma_{min})$ is derived based on \cite{HeKa}. This requires $\gamma_{min}$ to be smooth. In the noncompact case, however, we have to replace $\gamma_{min}$ by geodesic loops which do not close up smoothly.
\end{rem}
\begin{proof}
Let $C'$ be a bound on the absolute value of the sectional curvature of $g_J|_L$. Such a bound exists by the Gauss equation, the bound on $|Sec|$ and the bound on $\vectornorm{B}$. For any $p$ in $L,$ let $i_p$ be the radius of injectivity of $L$ at $p,$ and let $\ell_p$ be the length of the shortest (not necessarily closed) geodesic loop based at $p$. If no such geodesic exists, take $\ell_p=\infty$.  We have \cite[p. 178]{Pe}
\[
i_p\geq\min\left\{\frac{\pi}{\sqrt{C'}},\frac12\ell_p\right\}.
\]
It suffices to have a lower estimate for $\ell_p$. Let $p$ such that $\ell_p<\infty$ and let $\gamma$ be a geodesic loop based at $p$ realizing the length $\ell_p$. Since $\gamma$ is a geodesic loop in $L,$ its second fundamental form as a loop in $M$ is bounded by $H$. If $\ell_p\geq \frac{i_0}{4},$ we are done. So assume $\ell_p< \frac{i_0}{4}$. Let $\gamma$ be parameterized by arc length. Let $q=\exp_p\frac{-3i_0}{4}\gamma'(0)$. Let $f:B_{i_0}(q)\to[0,i_0)$ be given by $x\mapsto d(x,q;M).$ There is a function $F=F(n,C,i_0)$ such that $\vectornorm{\hess f}<F$ on $B_{i_0}(q)\setminus B_{i_0/2}(q)$ \cite[Ch. 10, Lemma 50]{Pe}. Let $h:[0,\ell_p]\to(i_0/2,i_0)$ be given by $h=f\circ\gamma$. Then
\[
\frac{dh}{dt}\Big|_{t=0}=g(\nabla f,\gamma'(0))=g(\nabla f,\nabla f)=1,
\]
and
\[
\frac{d^2h}{dt^2}(t)=\hess f(\gamma'(t),\gamma'(t))+g(\nabla f,B(\gamma'(t),\gamma'(t)))\geq -F-H.
\]
Since $f(0)=f(\ell_p)$, there is some intermediate point where the derivative of $h$ vanishes. Thus, $\ell_p\geq\frac1{F+H}$.
\end{proof}

\begin{lm}\label{lmModEst}
Let $\Sigma$ be a closed compact Riemann surface of genus greater than $1$ with its canonical metric $h$ of constant curvature $-1$. Let $\gamma$ be a simple closed geodesic in $\Sigma$, and let $I\subset\Sigma$ be doubly connected and such that the components of $\partial I$ are freely homotopic to $\gamma$. Then
\[
Mod(I)\leq \frac{1}{2\ell(\gamma;h)}.
\]
\end{lm}
\begin{proof}
Choose a lift $\tilde{\gamma}:[0,1]\to\mathbb{H}$ of $\gamma$ to the hyperbolic plane $\mathbb{H}$ by the universal covering map $\pi: \mathbb{H}
\to \Sigma$. Let $i:\mathbb{H}\to\mathbb{H}$ be a deck transformation taking $\tilde{\gamma}(0)$ to $\tilde{\gamma}(1)$. Let $A$ be the quotient of $\mathbb{H}$ by the isometries generated by $i$. It is immediate that $A$ is biholomorphic to an annulus, that $\pi$ induces a (non normal) covering map $A\to\Sigma$, that $\gamma$ lifts to a closed geodesic $\gamma'$ in $A$ and that $I$ lifts to a doubly connected subset $I'\subset A$ such that the components of $\partial I'$ are freely homotopic to $\gamma'$. Clearly,
\[
Mod(I)=Mod(I')\leq Mod(A).
\]
It thus suffices to estimate $Mod(A)$. For this, note that $A$ is a geodesic tubular neighborhood of $\gamma'$, so its metric is given in Fermi coordinates by $\frac{\ell(\gamma)^2}{4\pi^2}\cosh^2(\rho)d\theta^2+d\rho^2$. See~\cite[Theorem 4.1.1]{Bu}. By equation~\eqref{EqmodFrom}, we obtain
\[
Mod(A)=\int_{-\infty}^{\infty}\frac{d\rho}{2\pi\ell(\gamma)\cosh(\rho)}=\frac{1}{2\ell(\gamma)},
\]
which completes the proof.
\end{proof}

\begin{proof}[Proof of Theorem \ref{EnThTh}]
We start with the gradient inequality. When $M$ is compact, the gradient inequality  is proven for closed curves in \cite[Lemma 4.3.1]{MS2}. The proof there relies on the boundedness of the curvature and the derivatives of $J$ up to order 2. Therefore, the same applies whenever $\mathcal{F}$ satisfies Condition \ref{BoundCond2} in Definition \ref{BoundCond}.  Our formulation follows by conformal invariance of energy and of the expression  $\frac{d\mu(z)}{d\nu_{h}}r_{conf}^2(z)$. For curves with boundary and $M$ compact, a Hermitian metric $g$ is constructed in \cite{Frf, MS2} such that $L$ is totally geodesic and $\J TL=TL^{\perp}$. It is then shown that for any such metric, the gradient inequality holds for the measure $\mu_{u,g}$ with the constants depending on  curvature and of the derivatives of $J$ up to order 2. Since $M$ is compact, $g$ and $g_J$ are norm equivalent. Therefore the gradient inequality, with different constants, holds for the measure $\mu_u$. By Theorem \ref{tmLagTot} this generalizes to the noncompact setting when $M$ and $L$ satisfy the bounds appearing in Conditions \ref{BoundCond3} or \ref{BoundCond4} in Definition \ref{BoundCond}.

We now treat the cylinder inequality. Theorem~\ref{tmCylEq} and Lemma~\ref{lmInjRadEst} immediately imply the cylinder inequality with uniform constants whenever the conditions \ref{BoundCond1}, \ref{BoundCond2} or \ref{BoundCond3}, are satisfied. Note that for these cases we may take $c_2=2\pi$. We prove the remaining case. Let $\Sigma$ be a Riemann surface with boundary and let
\[
(u:(\Sigma,\partial\Sigma)\rightarrow (M,L))\in\mathcal{F}.
\]
For any clean and doubly connected $I\subset\Sigma_{\C}$ which meets only one connected component of $\partial\Sigma$ there is a connected component $L'\subset L$ such that $u(I\cap\partial\Sigma)\subset L'$. Thus, Theorem \ref{tmCylEq} applies with the same constants.  If $I$ meets two boundary components, $\gamma_1$ and $\gamma_2$, then $I\cap\Sigma$ is a strip. So $I$ is a cylinder which is embedded nontrivially in ${\Sigma_\C}$. We show that in this case the cylinder inequality holds vacuously because $Mod(I)$ is bounded above a priori. Let $h$ be the metric on $\Sigma_\C$ of Condition~\ref{BoundCond4} in Definition~\ref{BoundCond}. Let $\gamma$ be a minimizing geodesic freely homotopic to any boundary component of $I$. First consider the case that the curvature of $h$ vanishes. Let $I' \subset \Sigma_\C$ be an annulus with geodesic boundary such that $I \subset I'.$ Then we have
\[
Mod(I) \leq Mod(I') = \frac{2\pi Area(I';h)}{\ell(\gamma;h)^2} \leq \frac{2\pi}{\ell(\gamma;h)^2}=\frac{\pi}{2 d(\gamma_1,\gamma_2;h)^2}\leq \frac{\pi}{2\epsilon^2}.
\]
Otherwise, $h$ has negative curvature, and
\[
\ell(\gamma;h) \geq 2d(\gamma_1,\gamma_2;h)\geq2\epsilon.
\]
So by Lemma \ref{lmModEst}
\[
Mod(I)\leq \frac1{4\epsilon}.
\]
Thus, to cover both cases, we may take $c_2=\max\{\frac1{8\epsilon},\frac{\pi}{4\epsilon^2},2\pi\}$ and $c_3$ as in Theorem~\ref{tmCylEq}.
\end{proof}

\begin{proof}[Proof of Theorem~\ref{MainTheorem}]
This is just a rephrasing of a particular case of Theorem \ref{EnThTh}.
\end{proof}

\section{Proof of Theorem~\ref{tmLagTot}}\label{PrftmLagTot}
Let $g,h,$ be Riemannian metrics on $M$, and let $V,W,Z,$ be vector fields on $M$. Define tensors $H^{g,h}$ and $S^{g,h}$ by
\begin{align}
H^{g,h}(V,W)&:=\nabla^g_V W-\nabla^h_V W,
\\
S^{g,h}(V,W)Z&:=R^g(V,W)Z-R^h(V,W)Z.
\end{align}

\begin{lm}\label{lem:H}
Let $A\in Hom(TM,TM)$ be the tensor defined by
\begin{align}
 h(V,W):=g(AV,W),&& \forall p \in M,\quad V,W\in T_pM\notag.
\end{align}
Then
\begin{align}\label{eq:a}
&h(H^{g,h}(V_i,V_j),V_k)=\\&-\frac1{2}\{g((\nabla^{g}_{V_j}A)V_k,V_i)+g((\nabla^{g}_{V_i}A)V_k,V_j)
-g((\nabla^{g}_{V_k}A)V_i,V_j)\},
\notag\\
 &\forall p\in M,\quad V_i,V_j,V_k\in T_pM.\notag
\end{align}
\end{lm}

\begin{proof}
Let $p\in M$ and let $\{V_i\}$ be a basis of $T_pM$. Use $g$ and the basis $\{V_i\}$ to define geodesic normal coordinates on a neighborhood $N_p$ of $p$. Let $\{\overline{V_i}\}$ be the corresponding coordinate vector fields. Since this is a geodesic coordinate system centered at $p,$ we have $\nabla^{g}\overline{V_i}|_p=0$. Therefore,
\begin{align}
h(H^{g,h}(V_i,V_j),V_k)&=h(H^{g,h}(\overline{V_i},\overline{V_j}),\overline{V_k})|_p\\
&=h(\nabla^{g}_{{V_i}}\overline{V_j},\overline{V_k})-h(\nabla^{h}_{{V_i}}\overline{V_j},\overline{V_k})
\notag\\ &=-h(\nabla^{h}_{{V_i}}\overline{V_j},\overline{V_k})\notag.
\end{align}
By the Koszul formula,
\begin{equation}\label{eq:b}
h(\nabla^{h}_{{V_i}}\overline{V_j},\overline{V_k})=\frac1{2}\{V_jh(\overline{V_k},\overline{V_i})+
V_i h(\overline{V_k},\overline{V_j})-
V_kh(\overline{V_i},\overline{V_j})\}
\end{equation}
Now
\begin{align}
V_ih(\overline{V_j},\overline{V_k})&=V_ig(A\overline{V_j},\overline{V_k})\\&=
g(\nabla^{g}_{V_i}(A\overline{V_j}),\overline{V_k})+g(A\overline{V_j},
\nabla^{g}_{V_i}\overline{V_k})\notag\\
&= g(\nabla^{g}_{V_i}(A\overline{V_j}),\overline{V_k})\notag\\
&=g((\nabla^{g}_{V_i}A)\overline{V_j}+A\nabla^{g}_{V_i}\overline{V_j},
\overline{V_k})\notag\\&=g((\nabla^{g}_{V_i}A)\overline{V_j},\overline{V_k})\notag\\&= g((\nabla^{g}_{V_i}A){V_j},{V_k})\notag.
\end{align}
Substitution into equation~\eqref{eq:b} gives equation~\eqref{eq:a}.
\end{proof}

\begin{cy}\label{cyABounded}
Suppose $h$ and $g$ are norm equivalent. Then $\vectornorm{H^{g,h}}^g_k$ is bounded if $\vectornorm{A}_{k+1}^g$ is. Furthermore, $\vectornorm{H^{g,h}}^h_k$  is then bounded if and only if $\vectornorm{H^{h,g}}^g_k$ is. It follows that for $T$ an arbitrary tensor, $\vectornorm{T}_{k+1}^g$ is bounded if and only if $\vectornorm{T}_{k+1}^h$ is.
\end{cy}
\begin{proof}
If $h$ and $g$ are norm equivalent, then $A$ and $A^{-1}$ are bounded. Combine this observation with Lemma~\ref{lem:H} and straightforward calculation.
\end{proof}
\begin{lm}\label{cm:S}
We have
\begin{align}\label{eq:c}
S^{g,h}(V,W)Z &=(\nabla_V^hH^{g,h})(W,Z)-(\nabla_W^hH^{g,h})(V,Z)\\
&+H^{g,h}(V,H^{g,h}(W,Z))-H^{g,h}(W,H^{g,h}(V,Z)),\notag \\
&\qquad \qquad \qquad \qquad  \forall p \in M, \quad V,W,Z \in T_pM. \notag
\end{align}
\end{lm}
\begin{proof}
Let $\overline{V},\overline{W},\overline{Z},$ be coordinate vector fields in an $h$-geodesic coordinate chart centered at $p$ that extend $V,W,Z,$ respectively. Then
\begin{align}
\nabla^h_{\overline{V}}\nabla^h_{\overline{W}}\overline{Z}|_{p}&=\nabla^h_{\overline{V}}(\nabla^g_{\overline{W}}\overline{Z}-H^{g,h}(\overline{W},\overline{Z}))|_{p}\\
&=\nabla^h_{\overline{V}}\nabla^g_{\overline{W}}\overline{Z}-(\nabla_V^hH^{g,h})(W,Z)|_{p}\notag\\
&=\nabla^g_{\overline{V}}\nabla^g_{\overline{W}}\overline{Z}-H^{g,h}(\overline{V},\nabla^g_{\overline{W}}\overline{Z})-(\nabla_V^hH^{g,h})(W,Z)|_{p}\notag\\
&=\nabla^g_{\overline{V}}\nabla^g_{\overline{W}}\overline{Z}-H^{g,h}(V,H^{g,h}(W,Z))-(\nabla_V^hH^{g,h})(W,Z)|_{p}.\notag
\end{align}
Substitution into the standard formula for the curvature gives equation~\eqref{eq:c}.
\end{proof}

\subsection{The controlled reflection construction}
We pause for a moment to outline the next four subsections. Denote by $\phi:\nu_L\rightarrow M$ the map $(x,v)\mapsto \exp_x v$. In this subsection we describe two metrics induced by $g_J$ on a neighborhood of the zero section $O$ of $\nu_L$. The first one, $g_0,$ is non-linear, the pullback of $g_J$ by $\phi$, while the second, which we denote by $g_1,$ is linear. We then introduce the notion tameness of $L.$  This means that $g_0$ and $g_1$ are bounded with respect to one another in an appropriate sense and that $L$ has appropriately bounded geometry. Theorem~\ref{th:can} takes as input a tame Lagrangian $L$. Its output is a metric $h$ on a small neighborhood of $L$ satisfying the requirements of Theorem~\ref{tmLagTot} in that neighborhood. Furthermore, Theorem~\ref{th:can} provides estimates to control the geometry of $h$ in terms of the geometry of $g_J$. Theorem~\ref{TameCriterion} then provides an effective criterion for determining whether $L$ is tame. Its proof spans the next two subsections. In the last subsection it is shown that with the control provided by Theorem~\ref{th:can}, the metrics $h$ and $g_J$ can be interpolated in such a way that the curvature of the interpolated metric is bounded on $M$. The interpolated metric satisfies all the requirements of Theorem~\ref{tmLagTot}.

We start with the nonlinear metric. We first recall the definition of the cut locus $C(L)\subset M$. For $p\in L$ and $v\in \nu_{L,p},$ let $\gamma$ be the geodesic defined by $\gamma(t):=\phi(tv)$. A cut point of $L$ along $\gamma$ is a point $p=\phi(t_0v)$ such that $\gamma|_{[0,t_0]}$ has minimal length among all the curves connecting $\gamma(t_0)$ to $L$, but for all $t>t_0$, $\gamma|_{[0,t]}$ no longer has minimal length. The cut locus is the set of all cut points. For any $\delta>0$, let
\[
N_{\delta}:=B_\delta(L;g_J) \subset M, \qquad N'_{\delta}:=\phi^{-1}(N_{\delta})\subset \nu_L.
\]
Suppose $d(L,C(L);g_J) > \delta.$ Then $\phi|_{N'_{\delta}}$ is an embedding. To see this, note first that $\phi|_{N'_{\delta}}$ is injective. Indeed, suppose $\phi(tv_1)=\phi(sv_2)$ for $v_1,v_2\in N'_1$ and $t,s\in(0,\delta)$. Then $s=t$. Now apply the argument of \cite[Ch.~13, Proposition 2.2]{DC} to obtain injectivity.  By \cite[Ch.~11.4, Corollary 1]{BiCr64}, no focal point of $L$ along any geodesic can occur before the cut point. Thus, $\phi|_{N'_{\delta}}$ is an injective immersion of full dimension. Therefore, it is an embedding. We may thus define a metric $g_0$, by $g_0:=(\phi|_{N'_{\delta}})^*g_J$. We similarly define $J':=(\phi|_{N'_{\delta}})^*J.$

We now define the linear metric. Using $J$ to identify $TL$ and $\nu_L,$ we have a natural splitting
\[
T\nu_L\simeq TTL= V\oplus H
\]
into vertical and horizontal vectors with respect to the induced Levi-Civita connection on $L$. For any $o\in O$ and $x\in \pi^{-1}(o)$, this splitting induces a natural isomorphism
\[
C_{x}:T_x\nu_L\rightarrow T_o\nu_L=\nu_{L,o}\oplus T_oO.
\]
Namely, $C_{x}|_{V_x}$ is the canonical isomorphism
\[
V_x= T_x\nu_{L,o} = \nu_{L,o},
\]
and
\[
C_{x}|_{H_x}: = d\pi.
\]
For any $o\in O$, $x\in \pi^{-1}(o)$, and $v_1,v_2\in T_x\nu_L$, define
\begin{align}\label{eqg1com}
g_1(v_1,v_2):=g_0(C_{x}v_1,C_{x}v_2).
\end{align}

We say that $L$ is a \emph{$K$-tame Lagrangian} if the following conditions hold.
\begin{enumerate}
\item
$d(L,C(L)) \geq K^{-1}.$
\item
$g_0$ and $g_1$ are $K$-norm-equivalent on $N'_{K^{-1}}$. That is,
for any non-zero vector $v\in TN'_{K^{-1}}$, we have
\[
K^{-1}\leq \frac{\vectornorm{v}_{g_0}}{\vectornorm{v}_{g_1}}\leq K.
\]
\item
Let $D$ be the tensor on $N'_{K^{-1}}$ such that  $g_0(\cdot,\cdot)=g_1(D\cdot,\cdot)$. Then the covariant derivatives up to order two of $D$ are bounded by $K$. By  Corollary~\ref{cyABounded} and norm equivalence, it does not matter with respect to which of the metrics $g_0$ or $g_1$ we take the covariant derivative or measure its norm up to a redefinition of $K$.
\item
The second fundamental form of $L$ together with its first covariant derivative are bounded on $L$ by $K$.
\end{enumerate}

\begin {tm}\label{th:can}
Suppose that $\vectornorm{J}_2$ and $\vectornorm{R}_1$ are finite on $M$, and $L$ is $K$-tame. Then there is a $\delta\in (0,K^{-1}]$ and a Riemannian metric $h$ on $N_{\delta}$ with the following properties:
\begin{enumerate}
\item \label{it:1} For any $p\in N_{\delta}$,
\begin{align}
h(\J v,\J w)=h(v,w),&& \forall v,w\in T_pN_{\delta}.\notag
\end{align}
\item
\begin{align}
\J TL = TL^\perp. \notag&&
\end{align}
\item \label{it:tg} L is totally geodesic with respect to $h$\label{it:3}.
\item \label{th:canCond4} The curvature of $h$ is bounded on $N_{\delta}$.
\item \label{th:canCond5}
$h$ and $g_J|_{N_{\delta}}$ are norm equivalent. Furthermore, let $D$ be the unique automorphism of $TN_{\delta}$ satisfying $h(v,w):=g_J(v,Dw)$. Then $\vectornorm{D}^{g_J}_2$ is bounded on $N_{\delta}$.
\end{enumerate}
\end{tm}
Before proving Theorem~\ref{th:can} we state the following criterion for verifying the hypothesis of Theorem~\ref{th:can}.
\begin{tm}\label{TameCriterion}
Suppose that there is a $K>0$ such that
\[
\max\left\{\vectornorm{R}_2,\vectornorm{J}_3,\vectornorm{B}_2\right\}<K,
\]
and $L$ is $\frac1{K}$-Lipschitz. Then there exists $K' \geq K$ such that $L$ is $K'$-tame.
\end{tm}
The proof of Theorem~\ref{TameCriterion} will be given at the end of Subsection~\ref{SubSecDerNor}.

\begin{proof}[Proof of Theorem~\ref{th:can}.]
Let $N:=N_{K^{-1}}$ and $N':=N'_{K^{-1}}$. Define an almost complex structure $J_0$ on $N'$ by
\[
(J_0)_v:=C_{ v}^{-1}\circ J'\circ C_{ v}, \qquad v\in N'.
\]
We use the notation $\pi_v:TN'\rightarrow V$ and $\pi_h:TN'\rightarrow H$ for the vertical and horizontal projections respectively. Define a tensor \[
j:TN'\rightarrow TN'
\]
by
\[
j(y) :=-J_0\pi_h\J^{'}\pi_vy+\pi_hy.
\]

Define metrics $g_2$ and $g_3$ on $N'$ by
\[
g_2(\cdot,\cdot):=g_1(j\cdot,j\cdot),
\]
and
\[
g_3(\cdot,\cdot):=g_2(\J'\cdot,\J'\cdot).
\]

\begin{cm}\label{cm:tg}
$O$ is totally geodesic with respect to both $g_2$ and $g_3$.
\end{cm}

\begin{proof}
Observe that $O$ is totally geodesic with respect to $g_1$. We first show that for $i=2,3,$ and for any $w\in N'$,
\begin{equation}\label{g1g0smallo}
 g_i|_{H_{tw}}=g_1|_{H_{tw}}+O(t^2).
\end{equation}
Indeed, by definition, $g_2$ coincides with $g_1$ when restricted to $H$. It remains to prove~\eqref{g1g0smallo} for $g_3$. Note first that by smoothness of $J'$ and by the fact that $\J'|_O$ maps $H|_O$ to $V|_O$ and vice versa,
\[
 \J'|_{H_{tw}}=\pi_v\circ\J'|_{H_{tw}}+O(t), \qquad
 \J'|_{V_{tw}}=\pi_h\circ\J'|_{V_{tw}}+O(t).
\]
Moreover, tautologically,
\[
(\J' - \pi_v\circ\J')(H) \subset H, \qquad
(\J' - \pi_h\circ\J')(V) \subset V,
\]
and $H \perp V$ with respect to $g_2$ and $g_1$. Thus,
\begin{align}
 g_3(\cdot,\cdot)|_{H_{tw}}&=g_2(\J'\cdot,\J'\cdot)|_{H_{tw}}\notag\\
 &=g_2(\pi_v\J'\cdot,\pi_v\J'\cdot)|_{H_{tw}}+O(t^2)\notag\\
 &=g_1(J_0\pi_h\J'\pi_v\J'\cdot,J_0\pi_h\J'\pi_v\J'\cdot)|_{H_{tw}}+O(t^2)\notag\\
 &=g_1(\pi_h\J'\pi_v\J'\cdot,\pi_h\J'\pi_v\J'\cdot)|_{H_{tw}}+O(t^2)\notag.\\
 &=g_1(\cdot,\cdot)|_{H_{tw}}+O(t^2)\notag.
\end{align}

Let $\{x_1,...,x_n,y_1,...,y_n\}$ be a local coordinate system on $N'$ such that $\frac{\partial}{\partial x_i}|_O$ is tangent to $O$ and $\frac{\partial}{\partial y_i}|_O$ is perpendicular to $O$ with respect to $g_1$ for $i=1,...,n$. Then the same holds with respect to $g_2$ and $g_3$ because $g_1,g_2$ and $g_3$ coincide on $O.$ Since $O$ is totally geodesic with respect to $g_1$, we have
\begin{align}\label{eqTotGeoCrit}
\left. \frac{\partial }{\partial y_i}g_1\left(\frac{\partial}{\partial x_j},\frac{\partial}{\partial x_k}\right)\right|_O=0,\qquad i=1,...,n.
\end{align}
Since $H \perp V$ with respect to $g_2$ and $g_3|_{T_{tw}N'} = g_2|_{T_{tw}N'} + O(t),$ arguing as above,
\[
\left.g_l\left(\frac{\partial}{\partial x_j},\frac{\partial}{\partial x_k}\right)\right|_{tw}=\left. g_l\left(\pi_h\frac{\partial}{\partial x_j},\pi_h\frac{\partial}{\partial x_k}\right)\right|_{tw}+O(t^2), \qquad l=2,3.
\]
By equations~\eqref{g1g0smallo} and \eqref{eqTotGeoCrit}, we get
\[
\left. \frac{\partial }{\partial y_i}g_l\left(\frac{\partial}{\partial x_j},\frac{\partial}{\partial x_k}\right)\right|_{O}=0, \qquad i=1,...,n,~l=2,3.
\]
\end{proof}
We now define $h:=\phi_*(g_2+g_3)$. The metric $h$ obviously fulfills the first two conditions of Theorem \ref{th:can}. Claim~\ref{cm:tg} implies condition~\ref{it:tg}. We prove the last two conditions. Note first that since $L$ is $K$-tame, we have that for any tensor $A$ on $N'$, $\vectornorm{A}^{g_0}$ is bounded on $N'$ if and only if $\vectornorm{A}^{g_1}$ is. We wish to bound $\vectornorm{j}_2^{g_1}$ and the curvature of $g_1$. For this denote by $h_S$ the Sasaki metric on $TL$ where we consider $L$ with the induced metric. Denote by $\alpha:TL\to \nu_L$ the isomorphism $v\mapsto Jv$. Then $g_1=\alpha_* h_S.$ We utilize the following relations between the Levi-Civita connection of $h_S$ and the curvature of $L$ \cite{Kowalski71}. For a vector field $X$ on $L$, let $X^v$ and $X^h$ denote respectively the vertical and horizontal lifts of $X$ to vector fields on $TL$. Let $X$ and $Y$ be vector fields on $L$, let $\xi\in TL$, and let $p=\pi(\xi)$. We have
\begin{align}
\label{eqsas1}\nabla^{h_S}_{X^v}Y^v&=0,\\
\label{eqsas2}(\nabla^{h_S}_{X^h}Y^v)_\xi&=(\nabla_XY)_\xi^v+\frac12(R_p(\xi,Y_p)X_p)^h_\xi,\\
(\nabla^{h_S}_{X^v}Y^h)_\xi&=\frac12(R_p(\xi,X_p)Y_p)^h_\xi,\\
\label{eqsas4}(\nabla^{h_S}_{X^h}Y^h)_\xi&=(\nabla_XY)^h_\xi-\frac12(R_p(X_p,Y_p)\xi)^v_\xi.
\end{align}

\begin{cm}\label{cmjBound}
$\vectornorm{J'}_2^{g_1}$ and $\vectornorm{j}_2^{g_1}$ are bounded on $N'$
\end{cm}
\begin{proof}

By the assumption on $J$, Corollary~\ref{cyABounded} and tameness of $L,$ we have that $\vectornorm{J'}^{g_1}_2$ is bounded. To bound $\vectornorm{j}^{g_0}_2$ it thus suffices to bound $\vectornorm{\pi_h}^{g_1}_2$ and $\vectornorm{J_0}^{g_1}_2$ since $\pi_v=id-\pi_h$. Since $\alpha$ is an isometry between $(N',g_1)$ and $(TL,h_S),$ it suffices to bound pull-backs by $\alpha.$ By definition,
\begin{gather*}
\left(\alpha^*\pi_h\right)(X^h) = X^h, \qquad \left(\alpha^*\pi_h\right)(X^v) = 0, \\
\left(\alpha^*J_0\right)(X^h) = X^v, \qquad \left(\alpha^*J_0\right)(X^v) = -X_h.
\end{gather*}
Applying formulae~\eqref{eqsas1}-\eqref{eqsas4}, the covariant derivatives of $\alpha^*\pi_h$ and $\alpha^*J_0$ with respect to $h_S$ at a point $\xi\in TL$ involve only the curvature of $L$ contracted with $\xi$. The second covariant derivatives thus involve only the curvature of $L$ and its first derivative, again contracted with $\xi$. By the Gauss equation, the curvature on $L$ can be expressed in terms of the second fundamental form of $L$ and the curvature of $M$. Tameness of $L$ thus implies the claim.
\end{proof}

\begin{cm} \label{cmHnormEquiv}
For $\delta>0$ small enough, $\phi^*h$ is norm equivalent to $g_1$ on $N'_{\delta}$.
\end{cm}
\begin{proof}
By definition of $h$, it suffices that $g_2$ and $g_3$ are norm equivalent to $g_1$ on $N'_{\delta}$. For $\delta>0$ small enough, this follows for $g_2$ from the fact that $j$ is the identity when restricted to $O$, and from the bound on $\vectornorm{\nabla j}^{g_1}$ in Claim~\ref{cmjBound}. A similar argument applies to $g_3.$
\end{proof}

Let
\[
T:=j(id+J').
\]
It follows from Claim~\ref{cmjBound} that $\vectornorm{T}^{g_1}_2$ is bounded on $N'$. By the bounds on $B(\cdot,\cdot)$ and its first derivative, the curvature of the induced metric on $L$ and its first derivative are bounded on $L$.  By formulae~\eqref{eqsas1}-\eqref{eqsas4} it follows that $g_1$ has bounded curvature on $N'$. Let $T^t$ denote the transpose of $T$ with respect to $g_1.$ It follows from Lemma~\ref{cm:S}, Corollary~\ref{cyABounded} and Claim~\ref{cmHnormEquiv} that $h=\phi_*g_1(T^tT\cdot,\cdot)$ has bounded curvature on $N_{\delta}$ for $\delta>0$ small enough. This proves part~\ref{th:canCond4}.

We prove part~\ref{th:canCond5}. Tameness of $L$ and Claim~\ref{cmHnormEquiv} imply that $h$ and $g$ are norm equivalent on $N_{\delta}$. Let now $A$ be the tensor such that $g_1=g_0(\cdot,A\cdot)$. By tameness of $L$, $\vectornorm{A}^{g_0}_2$ is bounded on $N'$. We have $D=\phi_*A(T^tT)$. By Claim~\ref{cmjBound}, tameness of $L$ and Corollary~\ref{cyABounded}, we deduce that $\vectornorm{T^tT}^{g_0}_2$ is bounded. Thus $\vectornorm{D}^{g}_2$ is bounded as required. Both parts of condition~\ref{th:canCond5} are now proven.

\end{proof}

\subsection{Distance to the cut locus}
In this subsection, let $(M,g)$ be a Riemannian manifold, and let $L$ be a submanifold with second fundamental form $B$.

\begin{tm}\label{tmTubWidth}
Suppose there are $K,H,i_0,\epsilon>0,$ such that $|Sec_g|\leq K$, $InjRad(M;g)\geq i_0$, $\vectornorm{B}\leq H$, and $L$ is $\epsilon$-Lipschitz. Then
\[
d(L,C(L))>c
\]
for an a priori constant $c=c(K,H,i_0,\epsilon).$
\end{tm}
The proof of Theorem~\ref{tmTubWidth} will be given after Lemma~\ref{AngleEst} and its proof.
\begin{rem}
The quantitative dependence of $d(L,C(L))$ on the geometry of $(M,L)$ in the compact case appears in the literature \cite{Singh}. However, as in  the control of $InjRad(L)$ that we dealt with in Lemma~\ref{lmInjRadEst}, the proof there relies on compactness in an essential way. Namely, it utilizes the fact that the distance is realized, which fails in the non-compact case.
\end{rem}

In the following lemma, let $p,q\in L$, let $\gamma:[0,1]\rightarrow M$ be a minimizing geodesic connecting $p$ with $q,$ and let $\alpha: [0,\ell]\rightarrow L$ be a unit speed minimizing geodesic \textit{in $L$} connecting $p$ and $q$. Let $\theta$ be the angle between $\alpha'(0)$ and $\gamma'(0)$.
\begin{lm}\label{AngleEst}
There is a constant $c=c(K,H,i_0)$ such that $\theta\leq c\ell$.
\end{lm}
\begin{proof}
If $\theta = 0,$ the inequality holds trivially. So, assume $\theta > 0.$
Let $e_1,e_2 \in T_pM$ be orthonormal vectors such that
\[
e_1 = \dot \gamma(0)/|\dot\gamma(0)|, \qquad \dot \alpha(0) =  e_1\cos \theta +  e_2\sin \theta,
\]
for $\theta \in (0,\pi).$ Let $r  = \min\{i_0/3,\pi/6\sqrt{K}\}.$ Since $\theta < \pi,$ we may assume that $\ell < r$ if we make $c$ large enough. Let
\[
p_1 = \exp(-2re_1), \qquad p_2 = \exp(2re_2),
\]
and
\[
f_1(x) = d(p_1,x), \qquad f_2(x) = -d(p_2,x).
\]
The distance function $-f_2$ is convex on $B_{3r}(p_2)$ by~\cite[Theorem 27]{Pe} and the choice of $r.$ Since $\nabla f_2 = e_2$ is perpendicular to $\dot\gamma(0)$, we conclude that $f_2\circ\gamma$ has a critical point at $0$, which must be its unique maximum. In particular, $f_2(q)\leq -2r.$

On the other hand, write $\alpha_i = f_i\circ\alpha.$ Then
\[
\alpha_2'(0) = g(\alpha'(0),e_2) = \sin \theta > 0.
\]
Denote by $\alpha''$ the covariant derivative of $\alpha'$ with respect to $g.$ We have
\[
\alpha_2''(t) = g(\alpha''(t),\nabla f_2) + \hess f_2(\alpha'(t),\alpha'(t)).
\]
Since $\alpha$ is a geodesic in $L,$ we have
\[
|g(\alpha''(t),\nabla f_2)| = |g(B(\alpha'(t),\alpha'(t)),\nabla f_2)| \leq H.
\]
On the other hand, the estimate of~\cite[Theorem 27]{Pe} implies that
\[
|\hess f_2(\alpha'(t),\alpha'(t))| \leq \sqrt{K}\coth\left(\sqrt{K}r\right).
\]
Thus
\[
|\alpha_2''(t)| \leq H + \sqrt{K}\coth\left(\sqrt{K}r\right) =: f(i_0,H,K).
\]
Since $\alpha_2(\ell) = f_2(\gamma(1)) \leq f_2(\gamma(0)) = \alpha_2(0),$ Rolle's theorem implies that $\alpha'_2(t)$ vanishes for some $t \in (0,\ell).$ So,
\[
\sin \theta \leq \ell f(i_0,H,K).
\]

Thus, if $\theta$ is bounded away from $\pi$, say, $\theta\leq3\pi/4$, the claim holds with an appropriate constant $c$. If $\theta>3\pi/4$ there must be some intermediate point $t\in[0,\ell]$ where $\alpha_1(t)=\alpha_1(0)$. So, we may repeat the previous argument with $\alpha_1$ in place of $\alpha_2$ to obtain
\[
\theta\leq 2\pi|\cos(\theta)|\leq 2\pi f(i_0,H,K)\ell.
\]
\end{proof}

\begin{proof}[Proof of Theorem~\ref{tmTubWidth}]
Denote the focal locus of $L$ by $F(L)$. By \cite[Corollary 4.2]{Warner66}, there is a $\delta'=\delta'(K,H,i_0)>0$ such that $d(L,F(L))>\delta'$. Let
\[
r\in B_{\delta'}(L;g)\cap C(L).
\]
By the same argument as in \cite[Ch. 13, Prop. 2.2]{DC}, there are points $p, q\in L,$ and normal geodesics $\gamma_1$ and $\gamma_2$ connecting $p$ and $q$ respectively to $r$ and satisfying $\ell(\gamma_1)=\ell(\gamma_2)=d(r,L) = :\ell$. So, it suffices to bound $\ell$ from below. If $p=q$, we have $\ell\geq i_0$. So, assume $p\neq q$.

Let $\alpha$ and $\gamma_3$ be minimizing geodesics in $L$ and $M$ respectively, connecting $p$ and $q$. Let $\theta_1$ and $\theta_2$ be the angles between $\alpha$ and $\gamma_3$ at the endpoints. By Lemma~\ref{AngleEst} we have $\theta_1+\theta_2\leq 2c\ell(\alpha)=2cd(p,q;L)$. Since $L$ is $\epsilon$-Lipschitz we may assume without loss of generality that $d(p,q;L)\leq 1$, for otherwise $\ell\geq\epsilon/2$ and we are done. Similarly, since $InjRad(M)\geq i_0$, we may assume $p,$ $q$ and $r,$ are all contained in the geodesic ball centered at any one of them. Using again that $L$ is $\epsilon$-Lipschitz, we obtain
\begin{equation}\label{tubWidthEq1}
\theta_1+\theta_2\leq \frac{2c}{\epsilon}d(p,q;M)=\frac{2c}{\epsilon}\ell_3,
\end{equation}
where $\ell_3=\ell(\gamma_3)$.
For $i=1,2,3$, let $\delta_i$ be the angle opposite to $\gamma_i$ in the triangle formed by the points $p,q$ and $r$. Since $\gamma_i$ meets $\alpha$ perpendicularly at $p$ and $q$ for $i=1,2$, we have
\[
\theta_i+\delta_i\geq\pi/2.
\]

Let $S_k$ denote the $2$-dimensional simply connected manifold of constant curvature $k.$
Given a geodesic triangle in a metric space, a comparison triangle in $S_k$ is a triangle with the same side lengths in $S_k.$ The Rauch comparison theorem implies that the angles of the triangle $pqr$ are bounded above (resp. below) by the corresponding angles of a comparison triangle in $S_K$ (resp. $S_{-K}$).\footnote{See, e.g., section 4.1 of \cite{Karcher}. } By the Rauch upper bound on angles and the Gauss-Bonnet theorem,
\begin{equation}\label{tubWidthEq2}
\theta_1+\theta_2\geq \pi-\delta_1-\delta_2\geq\delta_3-\ell\delta_3K.
\end{equation}
On the other hand, let $p'q'r'$ be a comparison triangle in $S_{-K}$ and let $s$ be the midpoint of the segment $p'q'.$ By the Rauch lower bound on angles and the sine rule of hyperbolic geometry applied to the triangle $p'r's,$ we obtain
\begin{equation}\label{tubWidthEq3}
\frac{\sinh(\sqrt{K}\ell_3/2)}{\sinh{(\sqrt{K}\ell)}}\leq\sin(\delta_3/2)\leq \delta_3/2.
\end{equation}
Combining estimates~\eqref{tubWidthEq1},~\eqref{tubWidthEq2}, and~\eqref{tubWidthEq3}, we get
\begin{equation}\label{tubWidthEq4}
\frac{c}{\epsilon}\geq \frac{\sinh(\sqrt{K}\ell_3/2)}{\ell_3\sinh{(\sqrt{K}\ell)}}(1-\ell K)\geq\frac{\sqrt{K}}{2\sinh{(\sqrt{K}\ell)}}(1-\ell K).
\end{equation}
Inequality \eqref{tubWidthEq4} implies an estimate for $\ell$ from below as required.
\end{proof}

\subsection{Derivatives of the normal exponential map}\label{SubSecDerNor}
\subsubsection{Jacobi fields}
In the following, suppose $d(L,C(L))>0,$ and let $\eta<d(L,C(L))$. Write $N=N_{\eta},$ and $N'=N'_{\eta}$. For any $x\in N'$, let
\[
\gamma_x : [0,1] \to N'
\]
be the path defined by $\gamma_x(t):=tx$. For any path $\gamma:[a,b]\rightarrow N',$ denote by
\[
P_\gamma:T_{\gamma(b)}N'\rightarrow T_{\gamma(a)}N',
 \]
the parallel transport with respect to the Levi-Civita connection of the metric $g_0$. We define a tensor $A:TN'\rightarrow TN'$ by
\[
A_y:= C_{y}^{-1}\circ P_{\gamma_y}.
\]

\begin{cm}
We have
\[
g_0(v,w)=g_1(Av,Aw), \qquad \forall v,w \in T_xN' ,\quad x\in N'.
\]
\end{cm}
\begin{proof}
Indeed,
\begin{align}
g_0(v,w)&=g_0(P_{\gamma_x} v,P_{\gamma_x}w)\notag\\
&=g_0(C_xAv,C_xAw)\notag\\
&=g_1(Av,Aw).\notag
\end{align}
In the last transition, we use equation~\eqref{eqg1com}.
\end{proof}
For the rest  of this section, we omit superscripts from metric quantities when these are considered with respect to $g_0$.

Let $(X,g)$ be a Riemannian manifold and let $Y\subset X$ be a submanifold. For vectors $\nu$ normal to $Y$ and $v$ tangent to $Y$, let $S_{\nu}(v)$ be defined by
\[
g(S_{\nu}(v),\cdot) = g(B(v,\cdot),\nu).
\]
Let $\gamma$ be a geodesic in $X$ with $\gamma(0)\in Y$ and $\dot\gamma(0) \in T_{\gamma(0)}Y^\perp.$ A \textbf{$Y$-Jacobi} field $Z$ along $\gamma$ is a Jacobi field with initial conditions satisfying
\[
Z'(0)+S_{\gamma'}(Z(0))\in (T_{\gamma(0)}Y)^{\perp}.
\]

\begin{lm}\label{lmJacPar}
Let $x\in N'$, $v\in T_{\pi (x)}O,$ and let $\overline{v}$ be the parallel vector field along $\gamma_x$ extending $v.$ Then $\overline{u}:=A\overline{v}$ is an $O$-Jacobi field
with initial conditions
\[
\overline{u}(0)=v,
\]
\[
\overline{u}'(0)=\acs{J}'((\nabla_v\acs{J}')x-B(v,\acs{J}'x)).
\]
\end{lm}

\begin{proof}
Let $\beta$ be a path in $O$ such that $\beta(0)=\pi (x)$ and $\beta'(0)=v$. Introduce the notation $\nabla^L,\frac{D^{L}}{d}$, for the Levi-Civita connection of $g_0\big|_O$. Denote by $\tilde\xi(s)$ the parallel transport of the vector $J'x=J'\gamma_x'(0)$ along $\beta|_{[0,s]}$ with respect to $\nabla^L,$ and let $\xi(s) = -J'\tilde \xi(s).$ Define
\[
f(t,s):=\exp_{\beta(s)}t\xi(s).
\]
Since $g_0$ is the pull-back of $g_J$ by the exponential map, we have
\[
f(t,s) = \gamma_{\xi(s)}(t).
\]
It follows that
\[
d\pi\left(\frac{\partial f}{\partial s}(t,0)\right)=v, \qquad \frac{\partial f}{\partial s}(t,0)\in H_{\gamma_x(t)}.
\]
For the last assertion, note that $\xi$ is parallel with respect to the connection on $\nu_L$ induced by the isomorphism $\nu_L\simeq TL$.
By the definitions of $C$ and $A,$
\[
\frac{\partial f}{\partial s}(t,0)= C_{\gamma_x(t)}^{-1}v=A\overline{v}(\gamma_x(t)).
\]
In other words,
\[
\overline{u}(t)=\frac{\partial f}{\partial s}(t,0).
\]
For fixed $s$, the path $t \mapsto f(t,s)$ is a geodesic which is normal to $O$. So, by Lemma 4.1 in Chapter 10 of~\cite{DC}, $\overline u$ is an $O$-Jacobi field along $\gamma_x$. Moreover, we have  $\frac{\partial f(0,0)}{\partial s}=v$ and
\begin{align}
\restrict{\frac{D}{\partial t}\frac{\partial f}{\partial s}}{s,t=0}&=\restrict{\frac{D}{\partial s}\frac{\partial f}{\partial t}}{s,t=0}\\&=\restrict{\frac{D}{ds}\xi(s)}{s=0}\notag\\
&=\restrict{\frac{D}{ds}\xi(s)+\J'\frac{D^{L}}{ds}\left(\J'\xi(s)\right)}{s=0}\notag\\
&=\restrict{-\J'\frac{D}{ds}\left(\J'\xi(s)\right)+\J'\left(\frac{D}{ds}\J'\right)\xi(s)+\J'\frac{D^{L}}{ds}\left(J'\xi(s)\right)}{s=0}\notag\\
& =\acs{J}'\left((\nabla_v\acs{J}')x-B(v,\acs{J}'x)\right).\notag
\end{align}
\end{proof}
\begin{lm}\label{lmEstJacVert}
For $x\in N'$, let $v$ be orthogonal to $O$ at $\pi (x)$, and let $\overline{v}$ be the parallel vector field along $\gamma_x$ extending $v$. Let $\overline{u}(t)=tA\overline{v}(\gamma_x(t))$. Then $\overline{u}$ is the $O$-Jacobi field with initial conditions $\overline{u}(0)=0$ and $\overline{u}'(0)=v$.
\end{lm}
\begin{proof}
Let $\beta$ be the straight line $\beta(s)=x+sv$ in the vector space $(\nu_L)_{\pi( x)}$. Define $f(t,s)=\gamma_{\beta(s)}(t)$. We have
\begin{align*}
\frac{\partial f}{\partial s}(t,0)=tC_{\gamma_x(t)}^{-1}v=tC_{\gamma_x(t)}^{-1}\circ P_{\gamma_x|_{[0,t]}}\overline{v}(t)=tA\overline{v}({\gamma_x(t)}).
\end{align*}
\end{proof}

We proceed to estimate $A$ and its derivatives. For this we derive estimates for derivatives of $O$-Jacobi fields. Let $(X,g)$ be a Riemannian manifold and let $Y$ be a submanifold. Suppose $d(Y,C(Y))>\eta>0$, and let $B_{\eta}=B_{\eta}(Y;g).$ Let $n\in\mathbb{N}$ and $\epsilon>0$. Let $I_{\epsilon}=(-\epsilon,\epsilon)^{n+1}$ and  $I_{\eta,\epsilon}:= [0,\eta)\times I_{\epsilon}$. Let $t$ be the coordinate on $[0,\eta)$ and let $s=(s_0,...,s_n)$ be the coordinate on $I_{\epsilon}$. A family of normal geodesics is a smooth map $\Gamma:I_{\eta,\epsilon}\to B_{\eta}$ satisfying the following conditions for all $s\in I_{\epsilon}$:
\begin{enumerate}
\item $\Gamma(\cdot,s)$ is a geodesic and $\Gamma(\cdot,0)$ is of unit speed.
\item $\frac{\partial}{\partial t}\Gamma(0,s)\in TY^{\perp}$.
\item $\Gamma(0,s)\in Y$.
\end{enumerate}

Write
\begin{align}
Z^\Gamma_{i_0...i_m}(t,s)&:=\frac{D}{\partial s_{i_m}}...\frac{\partial}{\partial s_{i_0}}\Gamma(t,s),\notag\\
 W^\Gamma_{i_0...i_m}&:=\frac{D}{\partial t}Z_{i_0...i_m}^\Gamma.\notag
\end{align}
When there is no cause for confusion, we omit the superscript $\Gamma$. We abbreviate $I = (i_0,\ldots,i_m)$ and $|I| = m+1.$ It follows from Lemma~4.1 in Chapter~10 of~\cite{DC} that the fields $Z_i(\cdot,s)$ associated to a family of normal geodesics $\Gamma$ are $Y$-Jacobi fields.

The following lemma and its proof are adapted from \cite{Eichhorn91} which, in effect, treats the case where $Y$ is a point. What we must address is that in our case $Z_i(0)$ is only required to be tangent to $Y$ and may be non-zero.
\begin{lm}\label{JacobiEstimates}
Let $k\geq 0$ bound the absolute value of the sectional curvature on $N'.$ There is a $\delta=\delta(k)\in(0,\eta)$ and a smooth function
\[
 C_{n,k}:[0,\delta(k)]\times[0,\infty)\rightarrow [0,\infty)
\]
with the following significance. Let $t_0\in [0,\delta(k)]$. Suppose
\begin{equation}\label{InitCurvEst}
\max\left\{\{\vectornorm{Z_I(0,0)}\}_{|I|\leq n+1},\left\{\vectornorm{t_0W_I(0,0)}\right\}_{|I|\leq n+1},\vectornorm{R}_n\right\}\leq E,
\end{equation}
for some $E\in [0,\infty)$. Then for all $I$ with $|I| = n+1$, and all $t\in[0,t_0],$
\begin{align}
\max\{\vectornorm{Z_I(t,0)},\vectornorm{t_0 W_I(t,0)}\}\leq C_{n,k}(t,E).
\end{align}

\end{lm}
Before proving Lemma~\ref{JacobiEstimates}, we formulate two lemmas about derivatives of Jacobi fields. Let $\mathcal{X}(\Gamma)$ denote the space of smooth vector fields along $\Gamma$. Let $I=(i_0,i_1,...,i_n)$ and let $I'=(i_1,...,i_n)$. Let $A_{I'}, B_{I'}:\mathcal{X}(\Gamma)\to \mathcal{X}(\Gamma)$  denote the commutators
\[
A_{I'}:=\left[\frac{D^2}{\partial t^2},\frac{D}{\partial s_{i_n}}...\frac{D}{\partial s_{i_1}}\right],
\]
and
\[
B_{I'}:=\left[R\left(\cdot,\frac{{\partial}}{\partial t}\Gamma\right)\frac{{\partial}}{\partial t}\Gamma,\frac{D}{\partial s_{i_n}}...\frac{D}{\partial s_{i_1}}\right].
\]
Applying the operation $\frac{D}{\partial s_{i_n}}...\frac{D}{\partial s_{i_1}}$ to the Jacobi equation satisfied by $Z_{i_0}$, we obtain the inhomogeneous Jacobi equation
\begin{equation}
\frac{D^2}{\partial t^2}Z_I+R\left(Z_I,\frac{\partial}{\partial t}\Gamma\right)\frac{\partial}{\partial t}\Gamma=(A_{I'}+B_{I'})Z_{i_0}.
\end{equation}

\begin{lm}\label{lmInHomPart}
Let $R^{(m)}$ denote the $m^{th}$ covariant derivative of the curvature. Then $A_{I'}Z_{i_0}$ and $B_{I'}Z_{i_0}$ are linear combination of terms of the form $R^{(m)}$ or $R^{(j)}\circ R^{(l)}$ contracted with terms of the form $\frac{\partial}{\partial t}\Gamma,$ $Z_{J}$, and $W_{J}$ for $J\subset I$.
\end{lm}

\begin{proof}
As this claim does not refer to initial conditions, its proof is identical to that of the case dealt with in  \cite{Eichhorn91}. See Eq. 2.44 in \cite{Eichhorn91} and the discussion thereafter.
\end{proof}
Following  \cite{Eichhorn91}, we introduce a trigrading on the space of expressions of the type appearing in Lemma~\ref{lmInHomPart} as follows:
\[
\nabla\to(1,0,0),\qquad R\to(2,0,0),\qquad \frac{\partial}{\partial t}\Gamma\to (0,1,0),\qquad Z_i\to(0,0,1).
\]
Thus,
\[
R^{(m)}\to (2+m,0,0),\qquad Z_I\to (|I|-1,0,|I|),\qquad W_I\to (|I|-1,1,|I|).
\]
The tridegree is defined to be additive with respect to composition and contraction.
\begin{lm}\label{lmTrideg}
 $A_{I'}Z_{i_0}$ and $B_{I'}Z_{i_0}$ are homogeneous of tridegree $(2+n,2,1+n).$
\end{lm}
\begin{proof}
Again see the discussion after Eq. 2.44 in \cite{Eichhorn91}.
\end{proof}
Lemma~\ref{lmTrideg} implies that in the expressions of Lemma~\ref{lmInHomPart} we have $0\leq m\leq n$, $0\leq j+l\leq n-2$, and $|J|\leq n$. Furthermore, at most two terms of the form $W_J$ can appear in a given summand.

\begin{lm}\label{lmComparison}
There is a $\delta=\delta(k)$ with the following significance. Let $\gamma'$ be a geodesic of unit speed and let $V$ and $W$ be vector fields along $\gamma$ perpendicular to $\gamma$ satisfying the equation
\begin{equation}
\frac{D^2}{dt^2}V+R(V,\gamma')\gamma'=W.
\end{equation}
Let $\zeta:[0,\delta]\rightarrow \mathbb{R}$ be a continuous function such that
\[
\vectornorm{W(t)}+k\vectornorm{V(0)}\leq\zeta(t).
\]
Let $\xi$ be a solution of the inhomogeneous equation $\xi''-k\xi=\zeta$ with the initial conditions $\xi(0)=0$ and $\xi'(0)=\vectornorm{\frac{DV}{dt}(0)}$. Then
\begin{align}
&\vectornorm{V}\leq \xi+\vectornorm{V(0)}, &\vectornorm{\frac{D}{dt}V}\leq\xi'.
\end{align}
\end{lm}
\begin{proof}
When $\vectornorm{V(0)}=0$, this is Lemma 2.5 in \cite{Eichhorn91}. Otherwise, let $V_0$ be the parallel vector field along $\gamma$ satisfying $V_0(0)=V(0)$. Then $U:=V-V_0$ satisfies the inhomogeneous equation
\[
\frac{D^2}{dt^2}U+R(U,\gamma')\gamma'=W+R(V_0,\gamma')\gamma'
\]
with initial conditions $U(0)=0$ and $\frac{DU(0)}{dt}=\frac{DV(0)}{dt}$. The claim thus follows easily from the previous case.
\end{proof}

\begin{proof}[Proof of Lemma \ref{JacobiEstimates}]
Let $\gamma(t):=\Gamma(t,0)$. Let $V_I(t) :=Z_I(t,0)$ and let  $U_I(t):=(A_{I'}+B_{I'})Z_{i_0}(t,0)$. Then $V_{I}$ satisfies the inhomogeneous Jacobi equation
\begin{equation}\label{ecJacInhom}
\frac{D^2}{dt^2}V_I+R(V_{I},\gamma')\gamma'=U_I
\end{equation}
along $\gamma$. We prove the lemma by induction on $|I|$. The inhomogeneous Jacobi equation \eqref{ecJacInhom} splits into normal and tangent parts
\[
\frac{D^2}{dt^2}V_I^N+R(V^N_{I},\gamma')\gamma'=U_I^N,
\]
and
\[
\frac{D^2}{dt^2}V_I^T=\frac{d^2}{dt^2}g(V_I,\gamma')\gamma'=g(U_I,\gamma')\gamma'.
\]
Let
\[
f_1(t)=\cosh\left(\sqrt{k}t\right), \qquad f_2(t)=\sinh\left(\sqrt{k}t\right),
\]
\[
h_1(t,U_I^N)= -f_1 (t)\int_{0}^t\frac1{\sqrt{k}}f_2(s) \left(\|U_I^N(s)\| + k\|V_I^N(0)\|\right) ds,
\]
\[
h_2(t,U_I^N)= f_2 (t)\int_{0}^t\frac1{\sqrt{k}}f_1(s) \left(\|U_I^N(s)\| + k\|V_I^N(0)\|\right) ds,
\]
and
\begin{equation}\label{eq:hform}
h\left(t,U_I^N,a_1,a_2\right)=h_1(t,U_I^N) + h_2(t,U_I^N)+a_1f_1 (t) +a_2f_2 (t).
\end{equation}
Then $h$ is the general solution of the equation
\[
\xi'' - k \xi = \|U_I^N\| + k\|V_I^N(0)\|.
\]
Applying Lemma \ref{lmComparison}, we obtain
\begin{align}\label{ineqV}
\vectornorm{V_I^N(t)}\leq h\left(t,U^N_I,0,\frac1{\sqrt{k}}\vectornorm{\left(\frac{DV_I}{dt}(0)\right)^N}\right)+\vectornorm{V^N_I(0)},
\end{align}
and
\begin{align}\label{ineqDV}
\vectornorm{\left(\frac{DV_I}{dt}(t)\right)^N}\leq \frac{dh}{dt}\left(t,U^N_I,0,\frac1{\sqrt{k}}\vectornorm{\left(\frac{DV_I}{dt}(0)\right)^N}\right).
\end{align}
For the tangential part we get
\begin{align}\label{TangEq}
\vectornorm{V_I^T}&=|g(V_I,\gamma')|\\
&\leq\int_0^t\int_0^s\vectornorm{U^T_I(s')}ds'dt+t|g(V'_I(0),\gamma')|+|g(V_I^T(0),\gamma')|.\notag
\end{align}
By the inductive assumption, and Lemmas~\ref{lmInHomPart} and~\ref{lmTrideg}, we get
\begin{equation}\label{eq:bui}
\vectornorm{U_I(t)}\leq \frac{C'_{|I|-2,k}}{t_0^2}(t,E),
\end{equation}
for a universal function $C'_{|I|-2,k}$. The power of $t_0$ corresponds to the number of times $W_I$ can appear in $U_I.$ Estimate~\eqref{eq:bui}, assumption~\eqref{InitCurvEst} and the explicit formula~\eqref{eq:hform} for $h$, imply a bound as required on the right hand sides of inequalities~\eqref{ineqV},~\eqref{ineqDV} and~\eqref{TangEq}. This completes the inductive step and the proof of the claim.
\end{proof}

\begin{lm}\label{lmJacBasis}
Let $x\in B_{\eta}$ and let $t_0=d(x,Y)$. Let $\gamma:[0,t_0]\to B_{\eta}$ be the unit speed geodesic normal to $Y$ such that $\gamma(t_0)=x.$ Then $T_xB_{\eta}$ is spanned by vectors of the form $Z(t_0)$ where $Z$ is a $Y$-Jacobi field along $\gamma$. Furthermore, let $H>0$ bound the second fundamental form of $Y$ and let $k$ bound the absolute value of the sectional curvature of $X.$ Then there are constants $C=C(k,H)$ and $\delta'=\delta'(k,H)$ such that if $t_0\leq\delta'$ then
\[
\vectornorm{Z(0)}\leq C\vectornorm{Z\left(t_0\right)},
\]
and
\[
\vectornorm{\frac{D}{dt}Z(0)} \leq C \left(1+\frac1{t_0}\right)\vectornorm{Z\left(t_0\right)}.
\]
\end{lm}
\begin{proof}
Let $\nu_Y$ denote the normal bundle to $Y$, and let $\phi : \nu_Y \to X$ denote the map $(p,v) \mapsto \exp_p(v).$ The $Y$-Jacobi fields span the image of $d\phi.$ Since there are no focal points on $B_{\eta}$, we have that $\phi$ is a submersion. This gives the first part of the claim. The estimate on $\vectornorm{Z(0)}$ in the second part follows by comparison with hypersurfaces in constant curvature manifolds \cite[Theorem 4.3]{Warner66}. We explain the estimate on $\vectornorm{\frac{D}{dt}Z(0)}.$
Any $Y$-Jacobi field $Z$ splits as a sum $Z=Z_1+Z_2$ with
\begin{gather*}
Z_1(0)=0,\qquad Z_1'(0) \in (T_{\gamma(0)}Y)^{\perp}, \\
Z'_2(0)+S_{\gamma'}(Z_2(0))=0.
\end{gather*}
By the Rauch comparison theorem, we get
\[
\vectornorm{\frac{D}{dt}Z_1(0)} \leq C_1 \frac{\vectornorm{Z_1\left(t_0\right)}}{t_0},
\]
with $C_1$ depending on the curvature only. On the other hand, we get
\[
\vectornorm{\frac{D}{dt}Z_2(0)}\leq H \vectornorm{Z_2(0)}\leq HC \vectornorm{Z_2(t_0)}.
\]
Let $\theta(t)$ be the angle between $Z_1(t)$ and $Z_2(t)$. Then there is a constant $C'$ such that $\cos\theta(t)\leq C't$. Indeed,
\[
\left.\frac{d}{dt} g(Z_1(t),Z_2(t)) \right|_{t = 0} = 0.
\]
Using the Jacobi equation, we obtain the estimate
\[
\frac{d^2}{dt^2}g(Z_1(t),Z_2(t))\leq C''\vectornorm{Z_1(t)}\vectornorm{Z_2(t)}+\vectornorm{Z'_1(t)}\vectornorm{Z'_2(t)}.
\]
Applying Lemma~\ref{lmComparison}, one shows that for sufficiently small $t,$
\[
g_0(Z_1(t),Z_2(t))\leq C'''\vectornorm{Z_1'(0)}\vectornorm{{Z_2(0)}}t^2.
\]
We conclude the claim about
\[
\cos\theta(t)=\frac{g_0(Z_1(t),Z_2(t))}{\vectornorm{Z_1(t)}\vectornorm{Z_2(t)}}
\]
by applying the comparison of \cite[Theorem 4.3]{Warner66} to bound the denominator of the right hand side from below.
Taking $\delta'$ small enough, we have $\theta\leq\frac{3\pi}{4}$ for $t\in[0,\delta']$. In particular, for $i=1,2$,
\[
\vectornorm{Z_i(t_0)}\leq\sqrt2\vectornorm{Z(t_0)}.
\]
The claim now follows.
\end{proof}
\subsubsection{Constructing families of geodesics}
Below, we abbreviate $[m] := \{1,\ldots,m\}$ and write $I \vdash [m]$ to indicate that $I$ is a subset of $[m]$ with a chosen order on its elements.
For a smooth map $\Lambda:(-\epsilon,\epsilon)^m\to X$ and $I=\{i_1,...,i_l\}\vdash [m],$ we abbreviate
\[
D_I\Lambda:=\frac{D}{\partial s_{i_l}}\ldots\frac{\partial}{\partial s_{i_1}}\Lambda.
\]
and similarly for the covariant derivatives of sections of $\Lambda^*TX$ or maps with domain $[0,\eta) \times (-\epsilon,\epsilon)^m.$
For a family of normal geodesics $\Gamma^m,$ we abbreviate $Z^m_I : = Z^{\Gamma^m}_I$ and $W^m_I : = Z^{\Gamma^m}_I.$
\begin{lm}\label{lm:exG}
Suppose $\|R\|_{m-1} \leq E.$ Let $\Gamma^m:[0,\eta)\times (\epsilon,\epsilon)^m\to B_\eta$ be a family of normal geodesics such that
\[
\max_{I \vdash [m]}\{t_0\|W_I^m(0,0)\|,\|Z_I^m(0,0)\|\}\leq E.
\]
Let $Z$ be a vector field along $\Gamma$ such that $Z(\cdot,s_1,\ldots,s_m)$ is a $Y$-Jacobi field along the geodesic $\Gamma(\cdot,s_1,\ldots,s_m)$ and
\begin{equation}\label{eq:bdZ}
\max_{I \vdash [m]}\left\{\left\|D_I Z(0,0)\right\|,t_0\left\|D_I Z'(0,0)\right\|\right\}\leq E.
\end{equation}
Then there exists a family of normal geodesics
\[
\Gamma^{m+1}:[0,\eta)\times (-\epsilon,\epsilon)^{m+1}\to B_\eta
\]
such that
\begin{gather*}
\Gamma^{m+1}|_{\{s_{m+1}=0\}}=\Gamma^m,\qquad Z^{m+1}_{m+1}|_{\{s_{m+1}=0\}}=Z,\\
 \max_{I \vdash [m+1]}\{t_0\|W^{m+1}_I(0,0)\|,\|Z^{m+1}_I(0,0)\|\}\leq C(E).
\end{gather*}
\end{lm}
\begin{proof}
Let $\hat \Gamma^m : (-\epsilon,\epsilon)^{m+1} \to B_\eta$ be such that
\[
\hat \Gamma^m(s_1,\ldots,s_m,0) = \Gamma^m(0,s_1,\ldots,s_m)
\]
and
\[
\partial_{m+1}\hat\Gamma^m(s_1,\ldots,s_m,0) = Z(0,s_1,\ldots,s_m).
\]
Let $\hat W^m$ be a vector field along $\hat\Gamma^m$ such that
\[
\hat W^m(s_1,\ldots,s_m,0) = (\Gamma^m)'(0,s_1,\ldots,s_m)
\]
and
\[
D_{m+1}\hat W^m(s_1,\ldots,s_m,0) = Z'(0,s_1,\ldots,s_m).
\]
Let $\hat V^m$ be the projection of $\hat W^m$ to the normal bundle $\nu_Y$ and let
\[
\Gamma^{m+1}(t,s) = \exp_{\hat\Gamma^m(s)}(t\hat V(s)).
\]
Then by construction $\Gamma^{m+1}|_{\{s_m = 0\}} = \Gamma^m.$ The proof of Lemma 4.1 in Chapter 10 of~\cite{DC} shows $Z^{m+1}_{m+1} = Z.$ It remains to prove the bounds on $Z^{m+1}_I(0,0),W^{m+1}_I(0,0).$ By the bound $\|R\|_{m-1} \leq E$ and induction on $|I|,$ we may reorder the partial derivatives in $Z^{m+1}_I(0,0),W^{m+1}_I(0,0),$ to reduce the desired bounds to assumption~\eqref{eq:bdZ}.
\end{proof}
Let $\Lambda:(-\epsilon,\epsilon)^m \to Y,$ and write $\Dl{I}$ and $\Dr{I}$ for the induced connections on $\Lambda^*TY$ and $\Lambda^*{\nu_Y}$.
Let $a,b$ be natural numbers and let $c=a+b$.
Let $\xi_i$ be a section of
\[
\begin{cases}
\Lambda^*TY, & i=1\ldots a,\\
\Lambda^*\nu_Y,& i=a+1,\ldots c
\end{cases},
\]
and write
\[
\hat D_I\xi_j=\begin{cases}
\Dl{I}\xi_j, & j=1,\ldots a,\\
\Dr{I}\xi_j, & j= a+1\ldots c
\end{cases} .
\]
Let $T:TY^{\otimes a}\otimes \nu_Y^{\otimes b}\to TX|_Y$ be a smooth tensor. For $I,I_1,\ldots,I_c \vdash [m],$ we write $\uplus I_i = I$ if $\coprod I_i = I$ and the order induced on $I_i$ from $I$ is the chosen order on $I_i.$
\begin{lm}\label{lm:horr}
Suppose
\[
\max\left\{\|B\|_l,\max_{J \vdash [m],|J|\leq l+1}\|D_J\Lambda(0)\|,\|T\|_{l+1}\right\}<E.
\]
Then, for $|I|\leq l+1$ we have
\begin{multline*}
\left\|D_IT(\xi_1,\ldots,\xi_c)(0)-\sum_{\uplus I_i = I}T(\hat D_{I_1}\xi_1,\ldots,\hat D_{I_c}\xi_c)(0)\right\| \leq \\
\qquad\leq C_{|I|}(E,a,b)\sum_{j\in[c],|J|< |I|}\|\hat D_J\xi_j(0)\|.
\end{multline*}

\end{lm}
\begin{proof}
We prove the lemma by induction starting from $|I|=0$, in which case the claim is tautological. Assuming by induction the lemma holds for $|I|<N$, we prove the lemma for $|I|=N$. We prove first the case $T=\id$. For definiteness we treat the case $\xi = \xi_1$ is a section of $\Lambda^*TY$, the other case being similar. Let $I=(i_1,...,i_N)$ and $I'=(i_2,...,i_N)$. Thus
\[
D_I\xi=D_{I'}\hat D_{i_1}\xi+D_{I'}B(\partial_{i_1}\Lambda,\xi).
\]
We proceed to estimate both summands on the right hand side of the last expression. By induction with $T=\id$ and $\xi_1 = \hat D_{i_1}\xi$,
\[
\|D_{I'}\hat D_i\xi(0)-\hat D_I\xi(0)\|\leq C_{N-1}(E,1,0)\left(\sum_{|J|<N}\|\hat D_J\xi(0)\|\right).
\]
By induction with $T=B,$ $\xi_1=\partial_{i_1}\Lambda$ and $\xi_2=\xi,$ we obtain
\begin{align*}
\|D_{I'}B(\partial_{i_1}\Lambda,\xi)(0)\|&\leq\sum_{I_1 \uplus I_2 = I'}\|B(\hat D_{I_1}\partial_{i_1}\Lambda,\hat D_{I_2}\xi)(0)\| + \mbox{} \\
 &\qquad\qquad\mbox{}+C_{N-1}(E,2,0)\left(\sum_{j=1,2,|J|< |I'|}\|\hat D_J\xi_j(0)\|\right)\\
&\leq C'(E)\left(\sum_{|J|< N}\|\hat D_J\xi(0)\|\right).
\end{align*}
Combining the preceding estimates gives the lemma for $T=\id$.
We now prove the claim for general $T$. We have
\begin{multline*}
\left\|D_IT(\xi_1,\ldots,\xi_c)(0)-\sum_{\uplus I_i = I} T(D_{I_1}\xi_1,\ldots,D_{I_c}\xi_c)(0)\right\|\leq\\
\leq C''(E) \|T\|_{N}\left(\sum_{j\in[c],|I|<N}\|D_I\xi_j(0)\|\right).
\end{multline*}
From the case $T=\id$ we get
\[
\|D_I\xi(0)-\hat D_I\xi(0)\|\leq C(E,1,0)\sum_{|I|\leq l}\|\hat D_I\xi(0)\|.
\]
So,
\begin{multline*}
\| T(D_{I_1}\xi_1,\ldots,D_{I_c}\xi_c)(0)-T(\hat D_{I_1}\xi_1,\ldots,\hat D_{I_c}\xi_c)(0)\|\leq\\
\leq C'''(E)\|T\|_{N}\left(\sum_{j\in[c],|I|<N}\|\hat D_I\xi_j(0)\|\right).
\end{multline*}

\end{proof}

\begin{lm}\label{lm:InitEst}
Suppose $\max\{\|B\|_{m-1},\|R\|_{m-1}\} \leq E_1$ and $t_0 \leq \delta'(k,H)$ for $\delta'(k,H)$ as in Lemma~\ref{lmJacBasis}. Let $\Gamma = \Gamma^m$ be a family of normal geodesics satisfying the assumptions of Lemma~\ref{lm:exG} with $E = E_1$ and let $v \in T_{\Gamma(t_0,0)}X$ be a vector of unit length. Then there exists a vector field $Z$ along $\Gamma$ with $Z(t_0,0) = v$ that satisfies the assumptions of Lemma~\ref{lm:exG} with $E = C_1(E_1).$
\end{lm}
\begin{proof}
Abbreviate $\gamma(t): = \Gamma(t,0)$. By Lemma~\ref{lmJacBasis} there exists a $Y$-Jacobi field $\tilde Z$ along $\gamma$ such that $\tilde Z(t_0) = v$ and $\tilde Z(0),t_0\tilde Z'(0),$ are bounded. Let $\hat Z$ (resp. $\hat W$) be the vector field along $\Gamma|_{\{t = 0\}}$ obtained by parallel transport with respect to the induced connection on $Y$ (resp. $\nu_Y$) of the tangent vector $\tilde Z(0)$ (resp. the normal vector $\tilde Z'(0) + S_{\Gamma'(0,0)} \tilde Z(0)$) along lines parallel to coordinate axes in increasing order. Let
\begin{equation}\label{eq:dfU}
\hat U(s) = \hat W(s) - S_{\Gamma'(0,s)} \hat Z(s).
\end{equation}
Let $Z$ be the unique vector field along $\Gamma$ that satisfies the Jacobi equation along $\Gamma(\cdot,s)$ with initial conditions $Z(0,s) = \hat Z(s)$ and $Z'(0,s) = \hat U(s)$ for all $s \in (-\epsilon,\epsilon)^m.$
Then
\[
Z'(0,s) + S_{\Gamma'(0,s)} Z(0,s) = \hat U(s) + S_{\Gamma'(0,s)}\hat Z(s) = \hat W(s),
\]
which is a normal field by construction. So $Z$ is a $Y$-Jacobi field.

It remains to bound the partial derivatives of $Z$ at $0.$ Namely, for $I \vdash [m],$ we must bound $D_I\hat Z(0)$ and $D_I \hat U(0)$. First, we bound $\Dl{I}\hat Z(0)$ (resp. $\Dr{I}\hat W(0)$). Indeed, denote by $R^Y$ (resp. $R^{\nu_Y}$) the curvature of the induced connection on $TY$ (resp. $\nu_Y$). The bounds on $\|B\|_{m-1},\|R\|_{m-1},$ together with the Gauss and Ricci equations~\cite[Chapter 5]{DC} imply bounds on $\|R^Y\|_{m-1}$ and $\|R^{\nu_Y}\|_{m-1}.$ Using the curvature bounds and induction on $|I|$ to commute partial derivatives, we reduce to the case that $I$ is decreasingly ordered. In this case, $\Dl{I}\hat Z(0)$ (resp. $\Dr{I}\hat W(0)$) vanishes by construction of $\hat Z$ and $\hat W.$

To complete the proof, we apply Lemma~\ref{lm:horr} with $T = \id$ to bound $D_I\hat Z(0)$ and $D_I\hat W(0)$. We again apply Lemma~\ref{lm:horr} with $T = S$ to bound $D_I(S_{\Gamma'(0,\cdot)}\hat Z)(0)$. So, by equation~\eqref{eq:dfU}, we obtain a bound on $D_I\hat U(0).$
\end{proof}

\begin{lm}\label{ParEst}
Let $\Lambda:[0,\eta)\times(-\epsilon,\epsilon)^m \to X$ and let $V$ be a vector field along $\Lambda$ that is parallel along lines $s= const$. Let $1\leq k\leq m$. Suppose
\[
\max\left\{\|R\|_{k-1},\max_{J \vdash [m],|J|\leq k }\|D_JV(0,0)\|\right\}<E.
\]
Moreover, let $t_0\in [0,\eta)$ and suppose for all $t\in[0,t_0)$
\[
\max\left\{\max_{J \vdash [m],|J|\leq k}\{\|D_J\Lambda(t,0)\|,\|D_JD_t\Lambda(t,0)\|\}\right\}<E.
\]
Then $\max_{J \vdash [m],|J|\leq k}\|D_J V(t_0,0)\|\leq C_k(E,t_0).$
\end{lm}
\begin{proof}
We prove this by induction on $k$. When $k=0$ the claim is obvious. Suppose we have established the claim for all $k<N$. Let $I\vdash [m] $ with $|I|=N$. Start with the equation $D_I D_t V = 0$ and apply the commutation rule
\[
D_tD_iV-D_iD_tV=R(\partial_i\Lambda,\partial_t\Lambda)V,
\]
repeatedly to the expression $D_ID_t$ until the $t$ derivative comes last. It follows that $D_IV$ satisfies a differential equation of the form
\[
D_tD_I V=U
\]
with $U$ a linear combination of derivatives of curvature contracted with expressions of the form $D_J \Lambda,D_JD_t\Lambda,$ with $|J|\leq N$ and $D_J V$ with $|J|<N$. Thus $\|U\|\leq C(E,t_0)$. Combining this with the assumption on $\|D_I V(0,0)\|$, the induction step follows.
\end{proof}
\subsubsection{Main result}
We now return to the setting where $B_{\eta}=N',$ $g=g_0$ and $Y=O$.

\begin{lm}\label{AEst}
Assume $\eta\leq\min\{\delta'(k,H),\delta(k)\}$. Suppose
\[
\max\left\{\vectornorm{R}_m,\vectornorm{J}_{m+1},\vectornorm{B}_m\right\}<K.
\]
Then for some $C=C(K)$ we have $\vectornorm{A}_m<C.$
\end{lm}
\begin{proof}
We prove this by induction on $m$. The case $m=0$ is an immediate consequence of Lemmas~\ref{lmJacPar},~\ref{lmEstJacVert}, and~\ref{JacobiEstimates}. More specifically, in the case of Lemma~\ref{lmEstJacVert}, we apply Lemma~\ref{JacobiEstimates} to the Jacobi field $Z(t) = \frac{t}{\|x\|}A\overline{v}(t)$.

We proceed with the induction step. Denote by $A^{(i)}$ the $i^{th}$ covariant derivative of $A$. Suppose $\vectornorm{A^{(i)}}\leq C_i(K)$ for $0\leq i<m$. Let $x\in N'$ and let $v_0,v_1,...,v_m\in T_xN'$ be unit vectors.  We estimate  $A^{(m)}(v_0,...,v_m)$. Combining Lemmas~\ref{lm:exG}, and~\ref{lm:InitEst}, we inductively construct a family $\Gamma^m$ of normal geodesics which satisfies the assumption of Lemma~\ref{lm:exG} with $t_0 = \|x\|,$ the constant $E$ depending only on $K$, and such that $Z_i^m(\|x\|,0)= v_i$ for $i = 1,\ldots,m.$ Lemma~\ref{JacobiEstimates} implies that $Z_I$ is bounded in terms of $K$ for all $I\vdash [m]$.

Let $\hat V$ be the parallel vector field along $\gamma_x$ with $\hat V(\|x\|) = v_0.$ Let $v^\parallel$ and $v^\perp$ be the tangent and normal components of $\hat V(0)$. Let $\hat V^\parallel$ (resp. $\hat V^\perp$) be the vector field along $\Gamma^m|_{t =0}$ given by iterated parallel transport of $v^\parallel$ (resp. $v^\perp$) with respect to $\Dl{}$ (resp. $\Dr{}$) along coordinate axes in increasing order. As in the proof of Lemma~\ref{lm:InitEst}, it follows that $\|\Dl{I}\hat V^\parallel(0)\|$ and $\|\Dr{I}\hat V^\perp(0)\|$ are bounded in terms of $K$ for $|I| \leq m.$ Let $V^\parallel$ and $V^\perp$ be the vector fields along $\Gamma^m$ obtained by parallel transport of $\hat V^\parallel$ and $\hat V^\perp$ along lines $s=const$. Let $V=V^\parallel+V^\perp.$ Note that $V(\|x\|,0)=v_0$. From the construction of $V$ and Lemma~\ref{lm:horr}, we obtain an estimate on $D_IV(0,0)$ in terms of $K$ for all $I\vdash [m].$ Lemma~\ref{ParEst} then implies an estimate on $D_IV(t,0)$ in terms of $K$ for all $I\vdash [m].$

Let $\hat Z_0$ be the vector field along $\Gamma^m$ given by
\[
\hat Z_0(t,s)=A V^\parallel(t,s) + \frac{t}{\|x\|} A V^\perp(t,s).
\]
We show that $\hat Z_0$ satisfies the assumptions of Lemma~\ref{lm:exG} with $t_0 = \|x\|$ and the constant $E$ depending only on $K$.
Indeed, let $Q : TO \otimes \nu_O \to TN'$ be the tensor given by
\[
Q(v,x) = J'((\nabla_v J')x - B(v,J'x)).
\]
By Lemmas~\ref{lmJacPar} and~\ref{lmEstJacVert}, $\hat Z_0$ is an $O$-Jacobi-field along the geodesics $\Gamma(\cdot,s)$ with
\[
\hat Z_0(0,s) = V^\parallel(s), \qquad \hat Z'_0(0,s) = Q\left(V^\parallel(s),\Gamma'(0,s)\right) + \frac{1}{\|x\|} V^\perp.
\]
Lemma~\ref{lm:horr}, applied term by term with $T = \id$ or $T =Q$ as appropriate, gives the required bounds on the derivatives of $\hat Z_0$ at $(s,t) = (0,0).$
Let $\Gamma = \Gamma^{m+1}$ be an extension of $\Gamma^m$ as in Lemma~\ref{lm:exG} with $Z = \hat Z_0$. Write $s_0 := s_{m+1}.$ By Lemma~\ref{JacobiEstimates} we obtain that $Z_I$ is bounded in terms of $K$ for all $I\vdash\{0,\ldots,m\}$.

Note that $A^{(m)}(v_0,...,v_m)$ is the sum of the term
$D_{1\ldots m}AV(\|x\|,0)$
and terms of the form $A^{(i)}$ for $0\leq i< m$ contracted with expressions of the form $Z_{I}(\|x\|,0)$ and $D_JV(\|x\|,0)$ for $I\vdash\{0,\ldots,m\}$ and $J\vdash [m]$. But on the slice $t=\|x\|,$ we have $AV\equiv Z_0^\Gamma$. Therefore,
\[
D_{1\ldots m} AV(\|x\|,0)=Z_{0...m}(\|x\|,0).
\]
Thus,
\begin{align*}
\|A^{(m)}(v_0,...,v_m)\|&\leq C'(\{\|Z_I\|\}_{I\vdash \{0,...,m\}},\{\|D_JV\|\}_{J\vdash[m]},C_{m-1}(K))\\
&\leq C_m(K).
\end{align*}
This completes the induction step.
\end{proof}

\begin{proof}[Proof of Theorem~\ref{TameCriterion}]
By Lemma \ref{tmTubWidth} there is a $K'$ such that
\[
d(L,C(L))>\frac1{K'}.
\]
Since $A = \id$ on $O,$ by Lemma~\ref{AEst} with $m = 1,$ we deduce that $A$ is arbitrarily close to $\id$ on $N_\delta$ for small enough $\delta.$ Thus we can choose $\delta$ such that $g_0$ and $g_1$ are norm equivalent on $N_\delta.$ Let $D$ be such that $g_1(D\cdot,\cdot)=g_0(\cdot,\cdot)$. It remains to bound $\vectornorm{D}_2$. Note that $D=A^TA$ where the transpose is with respect to $g_0$. Indeed, we have $g_0(D^{-1}\cdot,\cdot)=g_1(\cdot,\cdot)=g_0(A^{-1}\cdot,A^{-1}\cdot)$, so $D^{-1}=(A^TA)^{-1}$. The theorem now follows from Lemma~\ref{AEst} with $m = 2.$
\end{proof}
\subsection{Interpolation of metrics}

Let $(M,g_1)$ be a Riemannian manifold and $L \subset M$ a submanifold. For $\delta> 0,$ write $U_1:=M\setminus B_{\delta}(L;g_1)$ and $U_2:=B_{2\delta}(L;g_1)$. Let
\[
A\in Hom(TU_2, TU_2)
\]
be positive definite and self adjoint with respect to $g_1$. Let $g_2$ be the metric on $U_2$ defined by
\[
g_2(v,w):=g_1(Av,w),\qquad p\in U_2,\quad v,w\in T_pU_2.
\]
\begin{lm}\label{lem:par}
Suppose $\vectornorm{R^{g_1}}^{g_1}\leq K<\infty$, the second fundamental form of $L$ is bounded with respect to $g_1$ and $d(L,C(L);g_1)>0.$ Suppose also that $\vectornorm{A}^{g_1}_2$ and $\vectornorm{A^{-1}}^{g_1}$ are finite. Then there exists $\delta > 0$ and a partition of unity $\{f,1-f\}$ subordinate to the cover $\{U_1,U_2\}$ such that $\vectornorm{f}^{g_1}_{2}$ is finite, and for the metric $h:=fg_1+(1-f)g_2$, we have $\vectornorm{R^h}^h$ is finite.
\end{lm}

\begin{proof}
Let $0< \delta < d(L,C(L);g_1)/2.$
By the boundedness of $A$ and $A^{-1}$, there is a $C\in[1,\infty)$ such that
\[
C^{-1}g_1(v,v)\leq g_2(v,v)\leq Cg_1(v,v)
\]
for any $p\in U_1\cap U_2$ and any $v\in T_pM$. Let $\{f,1-f\}$ be a partition of unity  subordinate to the cover $\{U_1,U_2\}$, and let  $h=fg_1+(1-f)g_2$. We have for $ p\in U_1\cap U_2$,
\[
\VN{R^h_p}{h}=\VN{{R^{g_1}_p+S^{g_1,h}_p}}{h} \leq C^2(\VN{R^{g_1}_p}{g_1} +
\VN{S^{g_1,h}_p}{g_1}).
\]
By Lemma \ref{cm:S}
\[
\VN{S^{g_1,h}}{g_1}\leq 2\left(\VNC{H^{g_1,h}}{g_1}{1}+\left(\VNC{H^{g_1,h}}{g_1}{1}\right)^2\right),
\]
and by Lemma \ref{lem:H},
\begin{align}
\VNC{H^{g_1,h}}{g_1}{1}&\leq 3\VNC{f\id +(1-f)A}{g_1}{2}{\vectornorm{(f\id+(1-f)A)^{-1}}^{g_1}_1} \\
&\leq 3\left(C_1\VNC{f}{g_1}{2}+\VNC{A}{g_1}{2}\right)^2\left({\vectornorm{(f\id +(1-f)A)^{-1}}^{g_1}}\right)^2\notag\\
&\leq 3 n^2 C^2 \left(C_1\VNC{f}{g_1}{2}+\VNC{A}{g_1}{2}\right)^2 \notag,
\end{align}
where $n = \dim M.$

It remains to construct $f$ such that $\VNC{f}{g_1}{2}$ is bounded on $U_1\cap U_2$. For this, let $\{k,1-k\}$ be a partition of unity of $[0,1]$ subordinate to the cover $\{[0,1),(0,1]\}$ and define
\[
f(p) = \begin{cases}
1-k\left(\frac{d(p,L;g_1)-\delta}{\delta}\right), & p\in U_1\cap U_2, \\
1, & p\in U_1\setminus U_2, \\
0, & p\in U_2\setminus U_1.
\end{cases}
\]
Since $U_1\cap U_2=B_{2\delta}(L;g_1)\setminus B_{\delta}(L;g_1)$, the function $f$ is continuous. The assumption $d(L,C(L);g_1)\geq 2\delta$ ensures that $f$ is smooth. To bound $\VNC{f^2}{g_1}{2}$ it sufficed to bound the derivatives of the distance function
\[
d(p)=d(p,L;g_1)
\]
up to order two. Note that $d$ satisfies the partial differential equation $|\nabla d|^2=1$, where $\nabla d$ is the gradient of $d$ with respect to $g_1$, which implies $\VNC{d}{g_1}{1}$ is bounded. It remains to bound the second derivatives of $d.$

For this let $\gamma:[0,2\delta]\to U_2$ satisfy $\gamma(0) \in L$ and $\gamma'(0) \in T_{\gamma(0)}L^\perp.$ Let $p=\gamma(t_0)$ for some $t_0\in(\delta,2\delta)$ and let $v\in \gamma'(t_0)^{\perp}\subset T_pU_1.$  Let $Z$ be an $L$-Jacobi field along $\gamma$ such that $Z(t_0)=v$.
Then
\[
\nabla_v\nabla d=\frac{D}{dt}Z(t_0).
\]
Lemmas~\ref{lmJacBasis} and~\ref{lmComparison} now provide and an estimate as required.
\end{proof}

\begin{proof}[Proof of Theorem~\ref{tmLagTot}]
By Theorem~\ref{TameCriterion}, we can apply Theorem \ref{th:can}. The metrics $g$ and $h$ of Theorem \ref{th:can} satisfy the conditions of Lemma \ref{lem:par}. Interpolating them as in Lemma \ref{lem:par} produces a Hermitian metric satisfying all the conditions of Theorem~\ref{tmLagTot}. The finiteness of $\vectornorm{J}^h_2$ follows from the bound on $\vectornorm{f}_2$ in Lemma~\ref{lem:par} and Corollary~\ref{cyABounded}.
\end{proof}

\bibliographystyle{amsabbrvc}
\bibliography{RefTCY1}
\vspace{.5 cm}
\noindent
Institute of Mathematics \\
Hebrew University, Givat Ram \\
Jerusalem, 91904, Israel \\

\end{document}